\documentclass[12pt,reqno,twoside]{amsart}
\usepackage{graphicx}
\usepackage{amssymb}
\usepackage{epstopdf}
\usepackage[asymmetric,top=3.5cm,bottom=4.3cm,left=3.1cm,right=3.1cm]{geometry}
\usepackage{booktabs} 
\usepackage{array} 
\usepackage{paralist} 
\usepackage{verbatim} 
\usepackage{subfig} 
\usepackage{tabularx}
\usepackage{amsmath,amsfonts,amsthm,mathrsfs,amssymb,cite}
\usepackage[usenames]{color}
\usepackage{bm}

\newtheorem{thm}{Theorem}[section]

\newtheorem{lem}{Lemma}[section]

\theoremstyle{definition}

\theoremstyle{remark}

\newtheorem{rem}{Remark}[section]
\numberwithin{equation}{section}

\newcommand{\bu}{\mathbf{u}}
\newcommand{\bv}{\mathbf{v}}

\newcommand{\bff}{\mathbf{f}}

\newcommand{\norm}[1]{\left\Vert#1\right\Vert}
\newcommand{\bh}{\mathbf{h}}
\newcommand{\bg}{\mathbf{g}}
\newcommand{\bvarphi}{\bm{\varphi}}
\newcommand{\bpsi}{\bm{\psi}}
\newcommand{\bnu}{\bm{\nu}}

\newcommand{\bx}{\mathbf{x}}
\newcommand{\by}{\mathbf{y}}
\newcommand{\bM}{\mathbf{M}}
\newcommand{\mch}{\mathcal{H}}
\newcommand{\bkappa}{\bm{\kappa}}
\newcommand{\bupsilon}{\bm{\upsilon}}
\newcommand{\bphi}{\bm{\phi}}

\newcommand{\Scal}{\mathcal{S}}
\newcommand{\Kcal}{\mathcal{K}}
\newcommand{\Pcal}{\mathcal{P}}
\newcommand{\Qcal}{\mathcal{Q}}

\newcommand{\Acal}{\mathcal{A}}
\newcommand{\Tcal}{\mathcal{T}}

\newcommand{\la}{\langle}
\newcommand{\ra}{\rangle}

\title[Novel Plasmonic Structures in Elasticity]{On Novel Elastic Structures Inducing Plasmonic Resonances With Finite Frequencies and Cloaking Due to Anomalous Localized Resonances}

\author{Hongjie Li}
\address{Department of Mathematics, Hong Kong Baptist University, Kowloon Tong, Hong Kong SAR.}
\email{hongjie$_{-}$li@yeah.net}

\author{Jinhong Li}
\address{School of Science, Qilu University of Technology, Jinan, Shandong 250353, P. R. China.}
\email{lijinhong@qlu.edu.cn}

\author{Hongyu Liu}
\address{Department of Mathematics, Hong Kong Baptist University, Kowloon Tong, Hong Kong SAR.\vspace*{-4mm}}
\address{\vspace*{-4mm}and}
\address{HKBU Institute of Research and Continuing Education, Virtual University Park, Shenzhen, P. R. China.}
\email{hongyu.liuip@gmail.com; hongyuliu@hkbu.edu.hk}

\begin{document}
\maketitle

\begin{abstract}

This paper is concerned with the theoretical study of plasmonic resonances for linear elasticity governed by the Lam\'e system in $\mathbb{R}^3$, and their application for cloaking due to anomalous localized resonances. We derive a very general and novel class of elastic structures that can induce plasmonic resonances. It is shown that if either one of the two convexity conditions on the Lam\'e parameters is broken, then we can construct certain plasmon structures that induce resonances. This significantly extends the relevant existing studies in the literature where the violation of both convexity conditions is required. Indeed, the existing plasmonic structures are a particular case of the general structures constructed in our study. Furthermore, we consider the plasmonic resonances within the finite frequency regime, and rigorously verify the quasi-static approximation for diametrically small plasmonic inclusions. Finally, as an application of the newly found structures, we construct a plasmonic device of the core-shell-matrix form that can induce cloaking due to anomalous localized resonance in the quasi-static regime, which also includes the existing study as a special case.

\medskip

\medskip



\end{abstract}

\section{Introduction}

\subsection{Mathematical setup}

We are concerned with the elastic wave scattering governed by the Lam\'e system in $\mathbb{R}^3$. Suppose that the space $\mathbb{R}^3$ is occupied by an isotropic elastic medium whose material property is characterized by a four-rank tensor $\mathbf{C}(\bx):=(\mathrm{C}_{ijkl}(\bx))_{i,j,k,l=1}^3$,
 \begin{equation}\label{eq:lame_constant}
 \mathrm{C}_{ijkl}(\bx):=\lambda(\bx)\bm{\delta}_{ij}\bm{\delta}_{kl}+\mu(\bx)(\bm{\delta}_{ik}\bm{\delta}_{jl}+\bm{\delta}_{il}\bm{\delta}_{jk}),\ \ \bx=(x_\alpha)_{\alpha=1}^3\in\mathbb{R}^3,
 \end{equation}
where $\lambda, \mu$ are real-valued functions, and $\bm{\delta}$ is the Kronecker delta. $\lambda$ and $\mu$ are the modulus of elasticity and are referred to as the Lam\'e parameters. For a regular elastic material, the Lam\'e parameters satisfy the following two strong convexity conditions,
 \begin{equation}\label{eq:convex}
 \mathrm{i)}.~~\mu>0\qquad\mbox{and}\qquad \mathrm{ii)}.~~3\lambda+2\mu>0.
 \end{equation}
 It is assumed that outside a bounded domain $D$, the elastic medium is uniform and homogeneous; that is,
 \begin{equation}\label{eq:homo1}
 \lambda=\lambda_0\quad\mbox{and}\quad \mu=\mu_0\quad\mbox{in}\ \ \mathbb{R}^3\backslash\overline{D},
 \end{equation}
where $\lambda_0$ and $\mu_0$ are two constants, particularly satisfying \eqref{eq:convex}. In order to describe the elastic wave propagation in the space $(\mathbb{R}^3; \lambda, \mu)$ mentioned above, we first introduce the Lam\'e operator
\begin{equation}\label{eq:lame}
\mathcal{L}_{\lambda,\mu} \mathbf{u}(\mathbf{x}):=\nabla\cdot\mathbf{C}\widehat{\nabla}\mathbf{u}(\mathbf{x})=\mu\Delta\mathbf{u}(\mathbf{x})+(\lambda+\mu)\nabla\nabla\cdot\mathbf{u}(\mathbf{x}),
\end{equation}
where $\mathbf{u}\in\mathbb{C}^3$ signifies the displacement field and $\widehat{\nabla}$ represents the symmetric gradient,
 \begin{equation}\label{eq:sg1}
 \widehat{\nabla}\mathbf{u}:=\frac{1}{2}\left(\nabla\mathbf{u}+\nabla\mathbf{u}^t \right)
 \end{equation}
 with the superscript $t$ denoting the matrix transpose. We consider our study in the frequency regime and let $\omega\in\mathbb{R}_+$ denote the frequency of the time-harmonic elastic wave. Let $\mathbf{f}\in H^{-1}(\mathbb{R}^3)^3$ denote a forcing/source term. Henceforth, we always assume that $\mathbf{f}$ is compactly supported in $\mathbb{R}^3$. Then the elastic wave scattering is described by the following Lam\'e system for $\mathbf{u}\in H_{loc}^1(\mathbb{R}^3)^3$,
\begin{equation}\label{eq:lame2}
\begin{cases}
& \mathcal{L}_{\lambda, \mu}\mathbf{u}(\mathbf{x})+\omega^2 \mathbf{u}(\bx)=\mathbf{f}\quad\mbox{in}\ \ \mathbb{R}^3,\medskip\\
& \mbox{$\mathbf{u}(\bx)$ satisfies the radiation condition. }
\end{cases}
\end{equation}
In \eqref{eq:lame2}, $\mathbf{u}$ is said to satisfy the radiation condition (or equivalently, to be radiating), if there holds the following Kupradze radiation condition as $|\mathbf{x}|\rightarrow+\infty$ (cf. \cite{Kup}),
\begin{equation}\label{eq:radiating}
\begin{split}
(\nabla\times\nabla\times\mathbf{u})(\bx)\times\frac{\bx}{|\bx|}-\mathrm{i}k_T\nabla\times\mathbf{u}(\bx)=&\mathcal{O}(|\bx|^{-2}),\\
\frac{\bx}{|\bx|}\cdot[\nabla(\nabla\cdot\mathbf{u})](\bx)-\mathrm{i}k_L\nabla\mathbf{u}(\bx)=&\mathcal{O}(|\bx|^{-2}),
\end{split}
\end{equation}
where $k_T:=\omega/\sqrt{\mu_0}$ and $k_L:=\omega/\sqrt{\lambda_0+2\mu_0}$. The strong convexity conditions in \eqref{eq:convex} guarantee the ellipticity of the PDO (partial differential operator) $\mathcal{L}_{\lambda,\mu}$ and hence the well-posedness of the scattering system \eqref{eq:lame2}.

In metamaterials, exotic elastic materials with negative material parameters have been artificially engineered; see e.g. \cite{KM} and \cite{LLBW}. In this paper, we consider the case that the negative elastic material is supported in $D$ of the following form
\begin{equation}\label{eq:parame_inside}
\widetilde\lambda=(\epsilon_1+\mathrm{i}\delta)\lambda_0\quad\mbox{and}\quad \widetilde\mu=(\epsilon_2+\mathrm{i}\delta)\mu_0,
\end{equation}
where $\epsilon_1$ and $\epsilon_2$ are two real scaling parameters and $\delta\in\mathbb{R}_+$ signifies the loss of the medium. The choice of $\epsilon_1$ and $\epsilon_2$ is critical and shall be investigated in details in what follows. In principle, $\epsilon_1$ and $\epsilon_2$ shall be chosen such that $\widetilde\lambda$ and $\widetilde\mu$ violate the convexity conditions in \eqref{eq:convex}. In such a case, $(D; \widetilde\lambda, \widetilde\mu)$ is referred to as a plasmonic inclusion. In what follows, we let $\widetilde{\mathbf{C}}$ signify the elastic tensor where the Lam\'e parameters are specified by $(\widetilde\lambda,\widetilde\mu)$ in $D$, and by $(\lambda_0,\mu_0)$ in $\mathbb{R}^3\backslash\overline{D}$, respectively. We also set $\mathbf{C}_0$ to signify the elastic tensor where the Lam\'e parameters are specified by $(\lambda_0, \mu_0)$ throughout $\mathbb{R}^3$. Now, let us consider the elastic wave scattering system \eqref{eq:lame2} associated with the medium configuration $\widetilde{\mathbf{C}}$, and denote by $\mathbf{u}_\delta$ the corresponding radiating field. Define,
\begin{equation}\label{eq:def0}
  E(\bu_{\delta},\mathbf{f}):=\delta\int_{D} \widehat{\nabla}\bu_{\delta}:\mathbf{C}_0\overline{\widehat{\nabla}\bu_{\delta}(\bx)}d\bx,
\end{equation}
where $\mathbf{A}:\mathbf{B}=\sum_{i,j}a_{ij}b_{ij}$ for two matrices $\mathbf{A}=(a_{ij})$ and $\mathbf{B}=(b_{ij})$. $E(\mathbf{u}_\delta,\mathbf{f})$ signifies the energy dissipation of the elastic system \eqref{eq:lame2} due to the presence of the plasmonic inclusion $(D;\widetilde\lambda,\widetilde\mu)$ and the source $\mathbf{f}$. The elastic configuration is said to be resonant if
\begin{equation}\label{eq:def1}
 \limsup_{\delta\rightarrow+0} E(\bu_{\delta},\mathbf{f})= +\infty.
\end{equation}
By \eqref{eq:def0} and \eqref{eq:def1}, it can be easily infer that if resonance occurs, then high oscillation would occur for the elastic wave field. If such a highly oscillatory behaviour of the resonant elastic wave field is confined within a certain bounded region, and outside of that region, the elastic field is convergent to a smooth one, then we say that cloaking due to anomalous localized resonance occurs.

\subsection{Background and motivation}

The mathematical research on plasmonic resonances has received significant attentions in recent years. The key ingredient is that due to the presence of negative material parameters, the ellipticity of the underlying PDO is broken. The non-elliptic PDO may then possess a nontrivial kernel and hence various resonance phenomena can be induced for an appropriate forcing term. Particularly, highly oscillatory behaviour of the resonant field is observed and hence energy blowup occurs accompanying the resonance. There are some other peculiar resonance behaviours have been found, such as the localizing effect; that is, the energy blowup of the resonant field is confined within a certain region and outside of that region, the resonant field remains bounded. The plasmonic resonances have been proposed for many striking potential applications including maging resolution enhancement, invisibility cloaking and energy harvesting. There have been intensive and extensive studies in the literature for optics and we refer to \cite{Ack13,Bos10,Brl07,CKKL,Klsap,LLL,GWM1,GWM2,GWM3,GWM4,GWM6,GWM7,GWM8} for the relevant results in electrostatics governed by the Laplace equation, \cite{ADM,AKL,KLO} in acoustics governed by the Helmholtz equation and \cite{ARYZ} in electromagnetism governed by the Maxwell system.

The plasmonic resonance associated to the Lam\'e system has been recently studied in \cite{AKKY2,DLL,LiLiu2d,LiLiu3d}. According to our earlier discussion, the first crucial ingredient is the appropriate choice of the plasmon Lam\'e parameters $\epsilon_1\lambda_0$ and $\epsilon_2\mu_0$ such that the following homogeneous Lam\'e system possesses nontrivial solutions,
\begin{equation}\label{eq:lame3}
\bm{\psi}\in H_{loc}^1(\mathbb{R}^3)^3,\quad \mathcal{L}_{\widehat\lambda,\widehat\mu} \bm{\psi}(\bx)=0, \quad \bm{\psi}(\bx)=\mathcal{O}(|\bx|^{-1})\ \ \mbox{as}\ \ |\bx|\rightarrow+\infty,
\end{equation}
where
\begin{equation}\label{eq:ps1}
(\widehat\lambda,\widehat\mu)=(\epsilon_1\lambda_0,\epsilon_2\mu_0)\chi_{D}+(\lambda_0,\mu_0)\chi_{\mathbb{R}^3\backslash\overline{D}},
\end{equation}
with $\chi$ denoting the characteristic function. The nontrivial solutions to \eqref{eq:lame3} are referred to as the perfect elastic waves in the literature which form the nontrivial kernel of the PDO $\mathcal{L}_{\widehat\lambda,\widehat\mu}$. In the aforementioned existing studies \cite{AKKY2,DLL,LiLiu2d,LiLiu3d}, the corresponding analysis heavily relies on a peculiar plasmonic structure of the form \eqref{eq:ps1} with $\epsilon_1=\epsilon_2=\epsilon$ being a proper negative number. That is, the plasmon parameters inside the inclusion are both proportional to the homogeneous background elastic parameters by a same negative constant. This assumption significantly simplifies the corresponding mathematical arguments in \cite{LiLiu2d,LiLiu3d} through the variational approach and in \cite{AKKY2,DLL} through the spectral approach via the use of the so-called Neumann-Poincar\'e operator. Indeed, the simplification is critical to the derivation of the perfect plasmon waves as well as the primal and dual variational principles for the Lam\'e system in \cite{LiLiu2d,LiLiu3d}; and to the reduction of the resonance study to an integral system associated with the Neumann-Poincar\'e operator in \cite{AKKY2,DLL}. This naturally raises a question: what is the appropriate definition of (plasmonic) negativity in elasticity? This question is easily understood in the context of optics, where two material parameters are involved, namely the electric permittivity and magnetic permeability. The positivity of the two optical parameters is also the mathematical condition to guarantee the ellipticity of the corresponding PDOs governing the relevant physical phenomena. Hence, if either one or both of the two optical parameters is (are) negative, then plasmonic resonances can be expected. This has actually been verified in the earlier mentioned literature. Hence, in the context of elasticity, it is unobjectionable to conjecture that for an elastic inclusion in \eqref{eq:ps1}, if $\epsilon_1$ and $\epsilon_2$ are such chosen that either one of the two convexity conditions in \eqref{eq:convex} is violated by the Lam\'e parameters $\epsilon_1\lambda_0$ and $\epsilon_2\mu_0$, then plasmonic resonance can be expected. We note that the existing studies in \cite{AKKY2,DLL,LiLiu2d,LiLiu3d} are a particular case of this conjecture where both the convexity conditions are violated. The first objective of the paper is to affirmatively verify this conjecture. In fact, we shall show that if either one of the two convexity conditions on the Lam\'e parameters is broken, then we can construct certain plasmon structures that induce resonances. Our construction is actually exclusive and includes the existing plasmonic structures as a special case. The second objective of the paper is concerned with the verification of the quasi-static approximation for the plasmonic resonance in elasticity. The plasmonic resonance has been experimentally discovered in the nanoscale. To study this phenomenon mathematically, one usually sets the frequency to be zero in the physical model based on the heuristic principle of quasi-static approximation. This approximate model has been adopted for the study in \cite{AKKY2,DLL,LiLiu2d,LiLiu3d}. In this paper, we shall work on the resonance in the finite-frequency regime associated with diametrically small plasmonic inclusions, and rigorously verify the validity of the quasi-static approximation. The third objective of our study is to consider the application of plasmonic resonances in invisibility cloaking.
As an application of the newly found structures, we construct a plasmonic device of the core-shell-matrix form that can induce cloaking due to anomalous localized resonance. Due to technical reasons, we only consider our study in this part in the quasi-static regime. Nevertheless, it still includes the localizing and cloaking study of plasmonic resonances in \cite{DLL} as a particular case.

Finally, we would like to remark two issues in our study. First, we confine our study within the spherical geometry. It might be innocently thought that working within the spherical geometry associated with piecewise constant elastic configuration, one might make use of the separating variable technique by conveniently manipulating the spherical wave series of the solutions in order to determine the plasmonic resonance. However, using the Fourier series technique, one would derive a large-scale system with many unknowns coupled together in many infinite series, it would be really hard to find a clue to determine the suitable plasmonic configurations that can induce resonance. As can be seen in our subsequent study, we shall make a combination use of layer potential technique, spectral analysis, asymptotic analysis as well as Fourier series expansion to derive all the possible plasmonic structures that can induce resonance in terms of energy blowup within the finite-frequency regime. Indeed, the spherical geometry is basically required in all of the existing stuides \cite{AKKY2,DLL,LiLiu2d,LiLiu3d} since essentially one would need accurate information on the perfect plasmon modes to \eqref{eq:lame3} in order to study the plasmonic resonance. In order to overcome this geometrical restriction, one might resort to the computer using numerical methods, but this is a matter beyond the scope of the present article.
Second, we are mainly concerned with the mathematical and theoretical study of plasmon materials in linear elasticity, and the engineering construction and physical implementation are beyond our scope. We hope that our theoretical findings may motivate some novel practical applications.

The rest of the paper is organized as follows. In Section 2, we collect some preliminary knowledge on quasi-static approximation and properties of layer potential operators, which shall be needed in the subsequent study. Section 3 is devoted to the main results on the novel elastic structures and the corresponding plasmonic resonances. In Section 4, we consider the cloaking due to anomalous localized resonance. Our study is concluded with some relevant discussion in Section 5.

\section{Preliminaries on the Lam\'e system}

\subsection{Quasi-static approximation}

We first fix a definite setup for our study.  As discussed earlier, we shall consider the plasmonic resonance for diametrically small elastic inclusion, namely, the Lam\'e system \eqref{eq:lame2}--\eqref{eq:radiating} associated with the elastic configuration $\widetilde{\mathbf{C}}$ defined after \eqref{eq:parame_inside}. Let $\Omega$ be a bounded domain with a connected complement and $\tau\in\mathbb{R}_+$ be an asymptotically small scaling parameter. Set $D=\tau\Omega$. Then by a straightforward change of variable $\by=\bx/\tau$, one can show that if $\mathbf{u}_\delta(\bx)$ satisfies the Lam\'e system \eqref{eq:lame2} associated with $\widetilde{\mathbf{C}}$ and $\mathbf{f}$, then $\bv_\delta(\by):=\bu_\delta(\bx)$ verifies the following Lam\'e system,
\begin{equation}\label{eq:lame4}
\begin{cases}
& \mathcal{L}_{\lambda, \mu}\mathbf{v}_\delta(\mathbf{y})+(\tau\omega)^2 \mathbf{v}_\delta(\by)=\mathbf{g}(\by)\quad\mbox{in}\ \ \mathbb{R}^3,\medskip\\
& \mbox{$\mathbf{v}_\delta(\by)$ satisfies the radiation condition, }
\end{cases}
\end{equation}
where $(\lambda,\mu)=(\widetilde\lambda,\widetilde\mu)\chi_\Omega+(\lambda_0,\mu_0)\chi_{\mathbb{R}^3\backslash\overline{\Omega}}$ and $\mathbf{g}(\by)=\tau^2\mathbf{f}(\tau \by)$. Hence, in order to study the resonance associated with $\widetilde{\mathbf{C}}$ and $\mathbf{f}$ for a diametrically small plasmonic inclusion, it suffices to consider the resonance associated with the Lam\'e system \eqref{eq:lame4}. Note that for an asymptotically small scaling parameter $\tau$, this is also equivalent to considering the resonance
for the Lam\'e system \eqref{eq:lame2}--\eqref{eq:radiating} associated with $\widetilde{\mathbf{C}}$ and $\mathbf{f}$ as $\omega\rightarrow+0$. Hence, in order to save notations, throughout the rest of the paper, we shall conduct our study for the Lam\'e system \eqref{eq:lame2}--\eqref{eq:radiating} for $\mathbf{u}_\delta$ associated with $\widetilde{\mathbf{C}}$ and $\mathbf{f}$ as $\omega\rightarrow+0$. As discussed in the introduction, we shall confine our study within spherical geometry; that is, the plasmonic inclusion $D$ is a central ball. However, at most of places in the subsequent study, we shall still stick to the notation $D$ of a general smooth domain. Indeed, most of the subsequent results hold for a general smooth domain. But there are some results which are only valid for a ball, and in such cases, we shall clearly indicate it. In doing so, some of the results can be directly used for future developments within general geometries.

\subsection{Properties on layer potential operator}
We present some preliminary knowledge on layer potential operators for the Lam\'e system. First, we introduce the fundamental solution to the PDO $\mathcal{L}_{\lambda_0,\mu_0} + \omega^2$, namely the Kupradze matrix $\mathbf{\Gamma}^{\omega}=(\mathbf{\Gamma}^{\omega}_{j,k})^3_{j,k=1}$ given by
\begin{equation}\label{eq:fundamentalsolution_w}
  \mathbf{\Gamma}^{\omega}_{j,k} = -\frac{\bm{\delta}_{jk}}{4\pi\mu|\bx|} e^{\frac{\mathrm{i}\omega|\bx|}{c_T}} + \frac{1}{4\pi\omega^2}\partial_j\partial_k \frac{e^{\frac{\mathrm{i}\omega|\bx|}{c_T}} - e^{\frac{i\omega|\bx|}{c_L}}}{|\bx|},
\end{equation}
with $\partial_j$ denoting $\partial / \partial \bx_j$, and
\begin{equation}
  c_T=\sqrt{\mu_0} , \quad c_{L}=\sqrt{\lambda_0 + 2 \mu_0}.
\end{equation}
From the expression \eqref{eq:fundamentalsolution_w} one can readily have the following expansion:
\begin{equation}
\begin{split}
   \mathbf{\Gamma}^{\omega}_{j,k}= & -\frac{1}{4\pi}\sum_{n=0}^{+\infty} \frac{\mathrm{i}^n}{(n+2)n!}\left( \frac{n+1}{c_{T}^{n+2}} + \frac{1}{c_{L}^{n+2}} \right)\omega^n\bm{\delta}_{jk}|\bx|^{n-1} \\
    & +\frac{1}{4\pi}\sum_{n=0}^{+\infty} \frac{\mathrm{i}^n(n-1)}{(n+2)n!}\left( \frac{1}{c_{T}^{n+2}} - \frac{1}{c_{L}^{n+2}} \right)\omega^n|\bx|^{n-3}x_j x_k.
\end{split}
\end{equation}
If $\omega=0$, then the fundamental solution to the PDO $\mathcal{L}_{\lambda_0,\mu_0}$, namely the Kelvin matrix $\mathbf{\Gamma}^{0}=(\mathbf{\Gamma}^{0}_{j,k})^3_{j,k=1}$ is given by
\begin{equation}\label{eq:fundamentalsolution_0}
  \mathbf{\Gamma}^0_{j,k}(\bx)=-\frac{\gamma_1}{4\pi} \frac{\delta_{jk}}{|\bx|} -\frac{\gamma_2}{4\pi} \frac{x_j x_k}{|\bx|^3},
\end{equation}
where
\begin{equation}
  \gamma_1=\frac{1}{2}\left( \frac{1}{\mu_0} + \frac{1}{2\mu_0+\lambda_0} \right) \quad \mbox{and} \quad \gamma_2=\frac{1}{2}\left( \frac{1}{\mu_0} - \frac{1}{2\mu_0+\lambda_0}\right).
\end{equation}
It can be verified that there holds the following relationship,
\begin{equation}\label{eq:expansion_gamma}
  \mathbf{\Gamma}^{\omega}= \mathbf{\Gamma}^{0} + \omega \bM^{\omega},
\end{equation}
where $\bM^{\omega}=(\bM^{\omega}_{jk})_{j,k=1}^3$ is given by
\begin{equation}
  \begin{split}
   \bM^{\omega}_{j,k}= & -\frac{1}{4\pi}\sum_{n=1}^{+\infty} \frac{\mathrm{i}^n}{(n+2)n!}\left( \frac{n+1}{c_{T}^{n+2}} + \frac{1}{c_{L}^{n+2}} \right)\omega^{n-1}\delta_{jk}|\bx|^{n-1} \\
    & +\frac{1}{4\pi}\sum_{n=2}^{+\infty} \frac{\mathrm{i}^n(n-1)}{(n+2)n!}\left( \frac{1}{c_{T}^{n+2}} - \frac{1}{c_{L}^{n+2}} \right)\omega^{n-1}|\bx|^{n-3} x_j x_k.
\end{split}
\end{equation}

Next, we introduce the layer potential operators. Given a distributional function $\bvarphi\in H^{-1/2}(\partial D)^3$, the single layer potential operator is defined as
\begin{equation}
  \Scal^{\omega}_{D}[\bvarphi](\bx) = \int_{\partial D} \mathbf{\Gamma}^{\omega}(\bx-\by)\bvarphi(\by)ds(\by), \quad \bx\in \mathbb{R}^3.
\end{equation}
 The conormal derivative (or the traction) of a function $\mathbf{u}\in H^1(D)$ on the boundary $\partial D$ is defined as
\begin{equation}\label{eq:def_traction}
  \frac{\partial \bu}{\partial\bnu} := \lambda_0(\nabla \cdot \bu)\bnu + \mu_0(\nabla \bu + \nabla \bu^t)\bnu,
\end{equation}
where $\bnu$ denotes the exterior unit normal vector to $\partial D$. Then the conormal derivative of the single layer potential enjoys the following jump relation (cf. \cite{Kup}),
\begin{equation}\label{eq:jump_single}
  \frac{\partial}{\partial \bnu} \Scal^{\omega}_D[\bvarphi]|_{\pm}(\bx)=\left( \pm \frac{1}{2}I + (\Kcal^\omega_D)^* \right) [\bvarphi](\bx), \quad \mbox{a.e.} \; \bx\in \partial D,
\end{equation}
where $(\Kcal^{\omega}_D)^*$ is the so-called Neumann-Poincar\'e (N-P) operator defined by
\begin{equation}\label{eq:operator_k_star}
  (\Kcal^{\omega}_D)^*[\bvarphi](\bx)=\mbox{p.v.}\quad\int_{\partial D} \frac{\partial}{\partial\bnu_\bx} \mathbf{\Gamma}^{\omega}(\bx-\by)\bvarphi(\by)ds(\by).
\end{equation}
In \eqref{eq:operator_k_star}, p.v. stands for the Cauchy principal value. Here and also what in follows, the subscripts $+$ and $-$ indicate the limits (to $\partial D$) from outside and inside $D$, respectively, and $\frac{\partial}{\partial\bnu_\bx} \mathbf{\Gamma}^{\omega}(\bx-\by)\bvarphi(\by)$ is defined by
$$\frac{\partial}{\partial\bnu_\bx} \mathbf{\Gamma}^{\omega}(\bx-\by)\bvarphi(\by):=\frac{\partial}{\partial\bnu_\bx} (\mathbf{\Gamma}^{\omega}(\bx-\by)\bvarphi(\by)).$$
When $\omega=0$, we denote $\Scal^{0}_{D}$ and $(\Kcal^{0}_{D})^*$ as $\Scal_{D}$ and $\Kcal^*_{D}$ for simplicity. From the decomposition of $\mathbf{\Gamma}^{\omega}$ in \eqref{eq:expansion_gamma}, one can obtain
\begin{equation}\label{eq:decom_single}
   \Scal^{\omega}_{D}= \Scal_{D} + \omega \Pcal^{\omega},
\end{equation}
with
\begin{equation}
  \Pcal^{\omega}[\bvarphi](\bx) =  \int_{\partial D} \bM^{\omega}(\bx-\by)\bvarphi(\by)ds(\by), \quad \bx\in \mathbb{R}^3.
\end{equation}
Furthermore, the operator $(\Kcal^{\omega}_D)^*[\bvarphi](\bx)$ could be decomposed as
\begin{equation}\label{eq:decom_K*}
  (\Kcal^{\omega}_D)^*=\Kcal^*_{D} + \omega \Qcal^{\omega},
\end{equation}
with
\begin{equation}
   \Qcal^{\omega}[\bvarphi](\bx)=\partial_{\bnu} \Pcal^{\omega}[\bvarphi](\bx).
\end{equation}
Let $\omega_1$ be a small positive number, then there exists a positive constant $C$ independent of $\omega<\omega_1$ such that
\begin{equation}\label{eq:217}
  \int_{D}\int_{\partial D}|\partial_\bx^{\alpha} \bM^{\omega}_{j,k}(\bx-\by)|^2ds(\by)d\bx\leq C,
\end{equation}
for all $\alpha=(\alpha_1,\alpha_2,\alpha_3)$ satisfying $|\alpha| \leq2$, where $\partial_\bx^{\alpha}$ is the partial derivative with respect to $\bx$. Then with the help of the Cauchy-Schwarz inequality and the last inequality \eqref{eq:217}, one can show that
\begin{equation}
  \norm{\Pcal^{\omega}[\bvarphi]}_{H^2(D)^3} \leq C \norm{\bvarphi}_{L^2(\partial D)^3},
\end{equation}
and
\begin{equation}
  \norm{\Qcal^{\omega}[\bvarphi]}_{H^1(D)^3} \leq C \norm{\bvarphi}_{L^2(\partial D)^3},
\end{equation}
for all $\bvarphi \in L^2(\partial D)^3$. By the trace theorem, $\Pcal^{\omega}$ maps $L^2(\partial D)^3$ into $H^{3/2}(\partial D)^3$ and $\Qcal^{\omega}$ maps $L^2(\partial D)^3$ into $H^{1/2}(\partial D)^3$. By duality,  $\Pcal^{\omega}$ maps $H^{-3/2}(\partial D)^3$ into $L^2(\partial D)^3$, and by interpolation, $H^{-1/2}(\partial D)^3$ into $H^1(\partial D)^3$. Similarly, one can show that $\Qcal^{\omega}$ maps $H^{-1/2}(\partial D)^3$ into $L^2(\partial D)^3$. To sum up, for a given $\omega_1$, there exists a positive constant $C$ independent of $\omega\leq \omega_1$ such that
\begin{equation}\label{eq:estimate_P}
  \norm{\Pcal^{\omega}[\bvarphi]}_{H^1(\partial D)^3} \leq C \norm{\bvarphi}_{H^{-1/2}(\partial D)^3},
\end{equation}
and
\begin{equation}\label{eq:estimate_Q}
  \norm{\Qcal^{\omega}[\bvarphi]}_{H^0(\partial D)^3} \leq C \norm{\bvarphi}_{H^{-1/2}(\partial D)^3},
\end{equation}
for all $\bvarphi \in H^{-1/2}(\partial D)^3$.

Next we define a new inner product on $H^{-1/2}(\partial D)^3$ by
\begin{equation}\label{eq:def_inner_h_star}
  (\bvarphi,\bpsi)_{\mch^*}:= -\la\bvarphi, \Scal_D[\bpsi]\ra, \quad \bvarphi,\bpsi \in H^{-1/2}(\partial D)^3,
\end{equation}
where $\la\cdot,\cdot\ra$ is the duality pairing of $H^{-1/2}(\partial D)^3$ and $H^{1/2}(\partial D)^3$. Since the operator $\Scal_D$ is invertible from $H^{-1/2}(\partial D)^3$ to $H^{1/2}(\partial D)^3$, hence the right side of \eqref{eq:def_inner_h_star} is well-defined. It can be verified that $(\cdot,\cdot)_{\mch^*}$ is the inner product on $H^{-1/2}(\partial D)$ and we denote by $\norm{\cdot}_{\mch^*}$ the norm induced by this inner product; see \cite{AJKKY15} for more relevant results. Furthermore, there holds
\begin{equation}
 \norm{\bvarphi}_{\mch^*} \approx \norm{\bvarphi}_{H^{-1/2}(\partial D)^3}.
\end{equation}
{In what follows, we denote the space $H^{1/2}(\partial D)^3$ by $\mch$.}

The N-P operator $\Kcal^*_{D}$ is not a compact operator even if the domain $D$ has a smooth boundary, thus we cannot infer directly that the N-P operator has point spectrum. However, in \cite{DLL},  when the domain $D$ is a ball, the eigenvalues of the N-P operator $\Kcal^*_{D}$ and its corresponding eigenfunctions have been explicitly derived.
\begin{thm}[\cite{DLL}]\label{thm:eigenfunction}
  Suppose that the domain $D$ is a central ball of radius $r_0$, then the eigenvalues of the operator $\Kcal^*_D$ are given by
  \begin{equation}\label{eq:eigenvalue}
    \begin{split}
      \xi^n_1= & \frac{3}{4n+2}, \\
      \xi^n_2= &\frac{3\lambda_0-2\mu(2n^2-2n-3)}{2(\lambda_0+2\mu_0)(4n^2-1)}, \\
      \xi^n_3= & \frac{-3 \lambda_0 + 2\mu_0(2n^2 + 2n - 3)}{2(\lambda_0 + 2\mu_0)(4n^2 - 1)},
    \end{split}
  \end{equation}
  where $n=1,2,3,\ldots$, and the corresponding eigenfunctions are given as follows,
  \begin{equation}\label{eq:eigenfunctions}
    \begin{split}
    \bkappa_1^{n,m}(\bx)= & \nabla_{\mathbb{S}}  Y^m_{n}(\hat{\bx}) \times \bnu_\bx,\\
    \bkappa_2^{n,m}(\bx)= & \nabla_{\mathbb{S}} Y^m_{n}(\hat{\bx}) +n  Y^m_{n}(\hat{\bx}) \bnu_\bx, \\
    \bkappa_3^{n,m}(\bx)= &-\nabla_{\mathbb{S}}  Y^m_{n-1}(\hat{\bx}) + nY^m_{n-1}(\hat{\bx})\bnu_\bx,
  \end{split}
  \end{equation}
 where $Y_n^m$, $m=-n,\ldots,n$, $n\in\mathbb{N}$, are the spherical harmonics of order $n$ and degree $m$, $\mathbb{S}$ denotes the unit sphere and $\nabla_{\mathbb{S}}$ is the surface gradient on $\mathbb{S}$.
\end{thm}

\begin{rem}\label{rem:21}
 It has been pointed out in \cite{DLL} that the eigenfunctions $\bkappa_i^{n,m}$ given in \eqref{eq:eigenfunctions} form an orthogonal basis on $L^2(\mathbb{S})^3$. By normalization, we can form an orthonormal basis with respect to the norm $\norm{\cdot}_{\mch^*}$ from those eigenfunctions, which we still denote by the same notations. That is, in what follows, we shall assume that
\begin{equation}
  \norm{\bkappa_i^{n,m}}_{\mch^*}=1,
\end{equation}
for $i=1,2,3$, $n\geq 1$ and $-n\leq m \leq n$. Moreover, these eigenfunctions are also the eigenfunctions of the operator $\Scal_{D}$ on the surface of the central ball $D$ of radius $r_0$. In fact, we have
\begin{equation}\label{eq:eigenfunctions_single}
  \begin{split}
      & \Scal_{D}[\bkappa_1^{n,m}] = e_{1}^n r_0 \bkappa_1^{n,m}, \\
      & \Scal_{D}[\bkappa_2^{n,m}] = e_{2}^n r_0 \bkappa_2^{n,m}, \\
      & \Scal_{D}[\bkappa_3^{n,m}] = e_{3}^n r_0 \bkappa_3^{n,m},
  \end{split}
\end{equation}
where
\begin{equation}\label{eq:coeff_eigenvalue_single}
  \begin{split}
      & e_{1}^n = -\frac{1}{\mu_0(2n+1)}, \\
      & e_{2}^n = -\frac{\mu_0(2+3n) + \lambda_0(n+1)}{\mu_0(\lambda_0 + 2\mu_0)(4n^2-1)}, \\
      & e_{3}^n = -\frac{\lambda_0(n-1)+\mu_0(3n-2)}{\mu_0(\lambda_0 + 2\mu_0)(4n^2-1)}.
  \end{split}
\end{equation}
\end{rem}
\begin{rem}\label{rem:22}
By direct calculations, it can be verified that the traction of $\Scal_{D}[\bkappa_i^{1,m}]$, $i=1,2$ and $m=-1,0,1$, vanishes on the boundary of $\partial D$. This fact shall be used in our subsequent arguments.
\end{rem}

\section{Novel plasmonic structures and resonances}\label{sect:resonance}

We are in a position to derive the novel plasmonic structures and study the associated resonances. In what follows, we let $\mathcal{L}_{\widetilde{\lambda}, \widetilde{\mu}}$, $\partial_{\widetilde{\bnu}}$, $\widetilde{\Scal}^{\omega}_D$ and $(\widetilde{\Kcal}^{\omega}_D)^*$ denote the Lam\'e operator, the associated conormal derivative, and the single layer potential operator and the N-P operator corresponding to the Lam\'e parameters $(\widetilde{\lambda}, \widetilde{\mu})$, respectively.

It can be readily shown that the elastic scattering system for $\mathbf{u}_\delta$ associated with $\widetilde{\mathbf{C}}$ defined after \eqref{eq:parame_inside} and a forcing source $\mathbf{f}$ is equivalent to the following transmission problem,
\begin{equation}\label{eq:problem_system}
  \left\{
    \begin{array}{ll}
      \mathcal{L}_{\widetilde{\lambda}, \widetilde{\mu}}\bu_{\delta}(\bx) + \omega^2\bu_{\delta}(\bx) =0,    & \bx\in D \\
      \mathcal{L}_{\lambda_0, \mu_0}\bu_{\delta}(\bx) + \omega^2\bu_{\delta}(\bx) =\bff, & \bx\in \mathbb{R}^3\backslash \overline{D}\\
      \bu_{\delta}(\bx)|_- = \bu_{\delta}(\bx)|_+,      & \bx\in\partial D \\
      \partial_{\widetilde{\bnu}}\bu_{\delta}(\bx)|_- = \partial_{\bnu}\bu_{\delta}(\bx)|_+, & \bx\in\partial D,
    \end{array}
  \right.
\end{equation}
where $\partial_{\bnu}$ is given in \eqref{eq:def_traction}.
Set
\begin{equation}
  \mathbf{F}(\bx):= \int_{\mathbb{R}^3} \mathbf{\Gamma}^{\omega}(\bx-\by)\bff(\by)d\by, \quad \bx\in \mathbb{R}^3.
\end{equation}
The solution $\bu_{\delta}$ to the system \eqref{eq:problem_system} can be represented as the following integral ansatz,
\begin{equation}\label{eq:solution}
  \bu_{\delta}=
 \left\{
   \begin{array}{ll}
     \widetilde{\Scal}_D^{\omega}[\bvarphi_{\delta}](\bx), & \bx\in D, \\
     \Scal_D^{\omega}[\bpsi_{\delta}](\bx) + \mathbf{F}, &  \bx\in \mathbb{R}^3\backslash \overline{D},
   \end{array}
 \right.
\end{equation}
for some $\bvarphi_{\delta}, \bpsi_{\delta}\in \mch^*$ that will be determined by using the transmission conditions across $\partial D$ in \eqref{eq:problem_system}. By the transmission conditions across $\partial D$, we have
\begin{equation}\label{eq:solution_system}
  \left\{
    \begin{array}{ll}
      \widetilde{\Scal}_D^{\omega}[\bvarphi_{\delta}] - \Scal_D^{\omega}[\bpsi_{\delta}] = \mathbf{F}, \\
      \partial_{\widetilde{\bnu}}\widetilde{\Scal}_D^{\omega}[\bvarphi_{\delta}]|_- - \partial_{\bnu}\Scal_D^{\omega}[\bpsi_{\delta}]|_+ = \partial_{\bnu}\mathbf{F} ,
    \end{array}
  \right.
  \quad \bx\in\partial D.
\end{equation}
Let $X:=\mch^{*}\times \mch^{*}$ and $Y:=\mch \times \mch^*$, and define an operator $\Acal_{\delta}^{\omega}:X\rightarrow Y$ by
\begin{equation}
  \Acal_{\delta}^{\omega} =
 \left[
   \begin{array}{cc}
      \widetilde{\Scal}_D^{\omega} & - \Scal_D^{\omega} \\
     \partial_{\widetilde{\bnu}}\widetilde{\Scal}_D^{\omega}|_- & - \partial_{\bnu}\Scal_D^{\omega}|_+ \\
   \end{array}
 \right].
\end{equation}
Then the system \eqref{eq:solution_system} can be rewritten as
\begin{equation}\label{eq:equation_system}
  \Acal_{\delta}^{\omega}
 \left[
   \begin{array}{c}
     \bvarphi_{\delta} \\
     \bpsi_{\delta} \\
   \end{array}
 \right]=
\left[
  \begin{array}{c}
    \mathbf{F} \\
    \partial_{\bnu}\mathbf{F} \\
  \end{array}
\right].
\end{equation}
Clearly, if the system \eqref{eq:equation_system} with $\delta=0$ is uniquely solvable, then no resonance occurs. Hence, in order to have resonance occurred, one needs to determine suitable $\widetilde\lambda$ and $\widetilde\mu$ defined in \eqref{eq:parame_inside} such that the system \eqref{eq:equation_system} with $\delta=0$ is not uniquely solvable. To that end, we next study the behaviour of the operator $(\Acal_{\delta}^{\omega})^{-1}$ as $\delta \rightarrow +0$.

First, we split $\Acal_{\delta}^{\omega}$ into two parts: $\Acal_{\delta}^{\omega}= \Acal_{\delta} + \Tcal_{\delta}^{\omega} $, where
\begin{equation}
  \Acal_{\delta}=
 \left[
   \begin{array}{cc}
     \widetilde{\Scal}_D & \Scal_D \\
     -1/2I+\widetilde{\Kcal}_D^* & -1/2I-\Kcal_D^* \\
   \end{array}
 \right].
\end{equation}
Then one can have from \eqref{eq:decom_single}, \eqref{eq:decom_K*}, \eqref{eq:estimate_P} and \eqref{eq:estimate_Q} that
\begin{equation}\label{eq:control_T_s}
  \norm{\Tcal_{\delta}^{\omega}}_{\mathcal{L}(X,Y)}\leq C {\omega}.
\end{equation}

We have the following lemma,

\begin{lem}\label{lem:inverse_A0}
 For $\bh\in \mch$ and $\bg\in \mch^*$ with the following Fourier series representations,
\begin{equation}\label{eq:bg}
  \bg =\sum_{i=1}^{3} \sum_{n=1}^{\infty} \sum_{m=-n}^{n}  g_i^{n,m} \bkappa_i^{n,m} ,
\end{equation}
and
\begin{equation}\label{eq:bh}
  \bh=\sum_{i=1}^{3}\sum_{n=1}^{\infty} \sum_{m=-n}^{n}  h_i^{n,m} \bkappa_i^{n,m},
\end{equation}
the solutions to
\begin{equation}\label{eq:zero_frequency_system}
  \Acal_{\delta}
 \left[
   \begin{array}{c}
     \bvarphi \\
     \bpsi \\
   \end{array}
 \right]=
 \left[
   \begin{array}{c}
     \bh \\
     \bg \\
   \end{array}
 \right]
\end{equation}
are given by
\begin{equation}\label{eq:solution_varphi}
   \bvarphi=\sum_{i=1}^{3} \sum_{n=1}^{\infty} \sum_{m=-n}^{n}  \varphi_i^{n,m} \bkappa_i^{n,m},
\end{equation}
with
\begin{equation}
  \varphi_i^{n,m}= \frac{g_i^{n,m}- (1/2 + \xi_i^{n})h_i^{n,m}/(r_0 e_i^n)}{{-1/2 + \widetilde{\xi}_i^{n} -(1/2+ \xi_i^{n})\widetilde{e}_i^{n}/e_i^{n}}},
\end{equation}
and
\begin{equation}\label{eq:solution_psi}
  \bpsi=\Scal_{D}^{-1}\widetilde{\Scal}_{D}[\bvarphi] - \Scal_{D}^{-1}[\bh].
\end{equation}
\end{lem}
\begin{proof}
 It is directly verified that \eqref{eq:zero_frequency_system} is equivalent to the following system of integral equations,
\begin{equation}\label{eq:system_without_frequency}
  \left\{
    \begin{array}{ll}
      \widetilde{\Scal}_{D}[\bvarphi]-\Scal_{D}[\bpsi] =\bh, \\
      (-1/2I+\widetilde{\Kcal}_D^*)[\bvarphi] + (-1/2I-\Kcal_D^*)[\bpsi]=\bg,
    \end{array}
  \right.
 \quad \mbox{on} \quad \partial D
\end{equation}
Since the operator $\Scal_{D}:$ $\mch^* \rightarrow \mch$ is invertible in three dimensions \cite{AJKKY15}, from the first equation in \eqref{eq:system_without_frequency} one can obtain
\begin{equation}\label{eq:rep_psi}
  \bpsi=\Scal_{D}^{-1}\widetilde{\Scal}_{D}[\bvarphi] - \Scal_{D}^{-1}[\bh].
\end{equation}
Substituting \eqref{eq:rep_psi} into the second equation of \eqref{eq:system_without_frequency} yields
\begin{equation}\label{eq:without_frequency}
  \left( -1/2I+\widetilde{\Kcal}_D^* - (1/2I + \Kcal_D^*)\Scal_{D}^{-1}\widetilde{\Scal}_{D} \right)[\bvarphi]=\bg-(1/2+\Kcal_D^*) \Scal_{D}^{-1}[\bh].
\end{equation}
Since the eigenfunctions given in \eqref{eq:eigenfunctions} are complete on $L^2(S)$, the density function $\bvarphi$ can be represented as follows
\begin{equation}
  \bvarphi=\sum_{i=1}^{3} \sum_{n=1}^{\infty} \sum_{m=-n}^{n}  \varphi_i^{n,m} \bkappa_i^{n,m}.
\end{equation}
Next we calculate the coefficient $\varphi_i^{n,m}$.
By Remark~\ref{rem:21}, since $\bkappa_i^{n,m}$ are also the eigenfunctions for the operator $\Scal_D$ which is invertible as mentioned above, with the representation of $\bh$ given in \eqref{eq:bh}, one can obtain from \eqref{eq:eigenfunctions_single} that
\begin{equation}\label{eq:bh_inves}
  \Scal_{D}^{-1}[\bh] = \sum_{i=1}^{3}  \sum_{n=1}^{\infty} \sum_{m=-n}^{n} \frac{h_i^{n,m}}{e_i^n r_0} \bkappa_i^{n,m} ,
\end{equation}
where $e_i^n$ is given in \eqref{eq:coeff_eigenvalue_single}. Substituting \eqref{eq:bg}, \eqref{eq:solution_varphi} and \eqref{eq:bh_inves} into \eqref{eq:without_frequency} and using \eqref{eq:eigenvalue} and \eqref{eq:eigenfunctions_single}, one can obtain that
\begin{equation}\label{eq:coeff_varphi_inm}
  \varphi_i^{n,m}= \frac{g_i^{n,m}- (1/2 + \xi_i^{n})h_i^{n,m}/(r_0 e_i^n)}{-1/2 + \widetilde{\xi}_i^{n} -(1/2+ \xi_i^{n})\widetilde{e}_i^{n}/e_i^{n}}.
\end{equation}
This completes the proof.
\end{proof}

As a consequence of Lemma \ref{lem:inverse_A0}, we obtain the following critical result.
\begin{thm}\label{cor_property_A}
  Let $(\bvarphi,\bpsi)$ be the solution of \eqref{eq:zero_frequency_system}. Then the following hold for sufficiently small $\delta$.
\begin{enumerate}
  \item $\norm{(\Acal_{\delta})^{-1}}_{\mathcal{L}(Y,X)}\leq  C\delta^{-1}$.

  \item If
\begin{equation}
  \epsilon_2=c_1^n,
\end{equation}
for $n\in\mathbb{N}$ and $n\geq 2$ with
\begin{equation}\label{eq:choice1}
  c_1^n:=-\frac{n+2}{n-1},
\end{equation}
then $\norm{\bvarphi}_{\mch^*}\geq C |\varphi_1^{n,m}|\delta^{-1}$.

If
\begin{equation}
  \epsilon_2=c_{2,1}^n,
\end{equation}
for $n\in \mathbb{N}$ with
\begin{equation}
  c_{2,1}^n:=-\frac{\epsilon_1 (n+1)\lambda_0}{(3n+2)\mu_0};
\end{equation}
or if
\begin{equation}
  \epsilon_2=c_{2,2}^n,
\end{equation}
for $n\in\mathbb{N}$ and $n\geq 2$ with
\begin{equation}
  c_{2,2}^n:=-\frac{(2n^2+1)\lambda_0+(2n^2+2n+2)\mu_0}{2(n-1)((n+1)\lambda_0+(3n+2)\mu_0)},
\end{equation}
then $\norm{\bvarphi}_{\mch^*}\geq C |\varphi_2^{n,m}|\delta^{-1}$.

 If
\begin{equation}
  \epsilon_1=c_3^n
\end{equation}
for $n\in\mathbb{N}$ with
\begin{equation}
  c_3^n:=-\frac{2\epsilon_2((n^2-n+1)\epsilon_2 + 3n^2+n-2)\mu_0}{((2n^2+1)\epsilon_2+2n^2-2)\lambda_0},
\end{equation}
then $\norm{\bvarphi}_{\mch^*}\geq C |\varphi_3^{n,m}|\delta^{-1}$.

\item If $\epsilon_2\neq c_1^n$, $\epsilon_1\neq c_3^n$ for all $n$, and $\epsilon_2\neq c_{2,j}^n$, for all $n$ and $j=1,2$, then $\norm{(A_{\delta})^{-1}}_{\mathcal{L}(Y,X)}\leq C$ for some $C$.
\end{enumerate}
\end{thm}

\begin{proof}
By straightforward though tedious calculations, one can show that
\begin{equation}
  \frac{1}{|-1/2 + \widetilde{\xi}_i^{n} -(1/2+ \xi_i^{n})\widetilde{e}_i^{n}/e_i^{n}| }\leq C \delta^{-1}.
\end{equation}
Moreover, from \eqref{eq:solution_varphi} there holds
\begin{equation}
\begin{split}
  \norm{\bvarphi}_{\mch^*}^2 & \leq C \delta^{-2} \sum_{i=1}^{3} \sum_{n=1}^{\infty} \sum_{m=-n}^{n}    |g_i^{n,m}- (1/2 + \xi_i^{n})h_i^{n,m}/(r_0 e_i^n)|^2 \\
    & \leq C \delta^{-2}(\norm{\bh}_{\mch}^2 + \norm{\bg}_{\mch^*}^2).
\end{split}
\end{equation}
Using \eqref{eq:solution_psi}, one can then obtain that
\begin{equation}
  \norm{\bpsi}_{\mch^*}^2 \leq C  \delta^{-2}(\norm{\bh}_{\mch}^2 + \norm{\bg}_{\mch^*}^2).
\end{equation}
Therefore we have shown $(1)$ in the theorem.

Next we give the proof of the result stated in $(2)$. If $i=1$ and $n=1$, one can show that
\begin{equation}
   \frac{1}{|-1/2 + \widetilde{\xi}_1^{1} -(1/2+ \xi_1^{1})\widetilde{e}_1^{1}/e_1^{1} |} \leq C.
\end{equation}
If $i=1$ and $n\geq 2$, by setting $\epsilon_2=c_1^n$, one has that
\[
 \frac{1}{|-1/2 + \widetilde{\xi}_1^{n} -(1/2+ \xi_1^{n})\widetilde{e}_1^{n}/e_1^{n} |}\geq C \delta^{-1}.
\]
When $i=2$ and $n=1$, only if $\epsilon_2=c_{2,1}^1$, there holds that
\[
\frac{1}{|-1/2 + \widetilde{\xi}_2^{1} -(1/2+ \xi_2^{1})\widetilde{e}_2^{1}/e_2^{1} |}\geq C \delta^{-1}.
\]
When $i=2$ and $n\geq 2$, either $\epsilon_2=c_{2,1}^n$ or $\epsilon_2=c_{2,2}^n$ can ensure the following inequality hold
\[
\frac{1}{|-1/2 + \widetilde{\xi}_2^{n} -(1/2+ \xi_2^{n})\widetilde{e}_2^{n}/e_2^{n} |}\geq C \delta^{-1}.
\]
When $i=3$ and $n\geq 1$, if $\epsilon_1=c_3^n$, then there holds
\[
\frac{1}{|-1/2 + \widetilde{\xi}_3^{n} -(1/2+ \xi_3^{n})\widetilde{e}_3^{n}/e_3^{n} |}\geq C \delta^{-1}.
\]
Therefore we have
\begin{equation}
  \norm{\bvarphi}_{\mch^*} \geq C |\varphi_i^{n,m}|\delta^{-1},
\end{equation}
where $i$ and $n$ depend on the values of $\epsilon_1$ and $\epsilon_2$. Hence, $(2)$ is proved.

If $\epsilon_2\neq c_1^n$, $\epsilon_1\neq c_3^n$ for all $n$, and $\epsilon_2\neq c_{2,j}^n$, for all $n$ and $j=1,2$, then one can obtain
\begin{equation}
  \frac{1}{|-1/2 + \widetilde{\xi}_i^{n} -(1/2+ \xi_i^{n})\widetilde{e}_i^{n}/e_i^{n} |}\leq C.
\end{equation}
Hence
\begin{equation}
\begin{split}
  \norm{\bvarphi}_{\mch^*}^2 & \leq C \sum_{i=1}^{3} \sum_{n=1}^{\infty} \sum_{m=-n}^{n}    |g_i^{n,m}- (1/2 + \xi_i^{n})h_i^{n,m}/(r_0 e_i^n)|^2 \\
    & \leq C (\norm{\bh}_{\mch}^2 + \norm{\bg}_{\mch^*}^2),
\end{split}
\end{equation}
and
\begin{equation}
  \norm{\bpsi}_{\mch^*}^2 \leq C  (\norm{\bh}_{\mch}^2 + \norm{\bg}_{\mch^*}^2).
\end{equation}
This proves $(3)$.

The proof is complete.
\end{proof}

 Next we give the estimate of $ E(\bu_{\delta})$ as $\delta\rightarrow +0$. By direct calculations and Green's formula, one has that
\begin{equation}
 \begin{split}
   E(\bu_{\delta}) & = E(\widetilde{\Scal}[\bvarphi_{\delta}]) =  \delta\int_{D} \widehat{\nabla}\widetilde{\Scal}_D^{\omega}[\bvarphi_{\delta}]:\mathbf{C}_0\overline{\widehat{\nabla}\widetilde{\Scal}_D^{\omega}[\bvarphi_{\delta}]}d\bx\\
     & = \Im\left( \int_{\partial D} \partial_{\bnu}\widetilde{\Scal}_D^{\omega}[\bvarphi_{\delta}]|_- \cdot \overline{\widetilde{\Scal}_D^{\omega}[\bvarphi_{\delta}]} ds(\bx)\right) - \int_{D} \mathcal{L}_{\lambda_0, \mu_0} \widetilde{\Scal}_D^{\omega}[\bvarphi_{\delta}] \cdot \overline{\widetilde{\Scal}_D^{\omega}[\bvarphi_{\delta}]}d\bx\\
     & =\Im\left( \int_{\partial D} \partial_{\bnu}\widetilde{\Scal}_D^{\omega}[\bvarphi_{\delta}]|_- \cdot \overline{\widetilde{\Scal}_D^{\omega}[\bvarphi_{\delta}]} ds(\bx)\right) + \omega^2 \int_{D} |\widetilde{\Scal}_D^{\omega}[\bvarphi_{\delta}]|^2 d\bx.
 \end{split}
\end{equation}
The last equality follows from
\begin{equation}
   \mathcal{L}_{\lambda_0, \mu_0} \widetilde{\Scal}_D^{\omega}[\bvarphi_{\delta}] + \omega^2 \widetilde{\Scal}_D^{\omega}[\bvarphi_{\delta}]=0.
\end{equation}
One can see from \eqref{eq:decom_single} and \eqref{eq:estimate_P} that
\begin{equation}
  \int_{D} |\widetilde{\Scal}_D^{\omega}[\bvarphi_{\delta}]|^2 d\bx\leq C \norm{\bvarphi_{\delta}}_{\mch^*}^2.
\end{equation}
By the jump relation \eqref{eq:jump_single}, one further has that
\begin{equation}
   \int_{\partial D} \partial_{\bnu}\widetilde{\Scal}_D^{\omega}[\bvarphi_{\delta}]|_- \cdot \overline{\widetilde{\Scal}_D^{\omega}[\bvarphi_{\delta}]} ds(\bx)=  \int_{\partial D} \left(- \frac{1}{2}I + (\widetilde{\Kcal}^\omega_D)^*\right)[\bvarphi_{\delta}] \cdot \overline{\widetilde{\Scal}_D^{\omega}[\bvarphi_{\delta}]}ds(\bx).
\end{equation}
From \eqref{eq:decom_single} and \eqref{eq:decom_K*}, one can find that
\begin{equation}\label{eq:energy_decom}
  \int_{\partial D} \partial_{\bnu}\widetilde{\Scal}_D^{\omega}[\bvarphi_{\delta}]|_- \cdot \overline{\widetilde{\Scal}_D^{\omega}[\bvarphi_{\delta}]}ds(\bx)= \int_{\partial D} \left(- \frac{1}{2}I + \widetilde{\Kcal}_D^*\right)[\bvarphi_{\delta}] \cdot \overline{\widetilde{\Scal}_D[\bvarphi_{\delta}]} ds(\bx) +\mathcal{A},
\end{equation}
where
\begin{equation}
\begin{split}
 \mathcal{A}=& \int_{\partial D} \left(\omega \widetilde{\Qcal}^{\omega}[\bvarphi_\delta]\right) \cdot \overline{\widetilde{\Scal}_D[\bvarphi_{\delta}]} ds(\bx)\\
 & + \int_{\partial D} \left(- \frac{1}{2}I + \widetilde{\Kcal}_D^* +\omega \widetilde{\Qcal}^{\omega}\right)[\bvarphi_{\delta}]\cdot \overline{\left(\omega \widetilde{\Pcal}^{\omega}[\bvarphi_{\delta}]\right) }ds(\bx).
\end{split}
\end{equation}
Since $\widetilde{\Kcal}_D^*$ is bounded on $H^{-1/2}(\partial D)^3$, one can deduce as follows
\begin{equation}\label{eq:energy_control_A}
 \begin{split}
   |\mathcal{A}| \leq & C \omega \Big( \norm{\widetilde{\Scal}_D[\bvarphi_{\delta}]}_{H^{-1/2}} \norm{ \widetilde{\Qcal}^{\omega}[\bvarphi_\delta]}_{H^{1/2}} + \\
   & \norm{\widetilde{\Pcal}^{\omega}[\bvarphi_{\delta}]}_{H^{1/2}}\left(C\norm{\bvarphi_\delta}_{H^{-1/2}} +\norm{\omega \widetilde{\Qcal}^{\omega}[\bvarphi_{\delta}]}_{H^{-1/2}}\right)    \Big ) \\
    \leq & C \omega\norm{\bvarphi_\delta}_{H^{-1/2}}^2.
 \end{split}
\end{equation}
Using the fact that the eigenfunctions given in Theorem \ref{thm:eigenfunction} are complete on $\mch^*$, the function $\bvarphi_{\delta}$ has the following Fourier representation
\begin{equation}
  \bvarphi_{\delta} =  \sum_{i=1}^{3}\sum_{n=1}^{\infty} \sum_{m=-n}^{n} \varphi_{\delta,i}^{n,m} \bkappa_i^{n,m}.
\end{equation}
Then one can conclude that
\begin{equation}
 \begin{split}
     &\Im\left( \int_{\partial D} \left(- \frac{1}{2}I + \widetilde{\Kcal}_D^*\right)[\bvarphi_{\delta}] \cdot \overline{\widetilde{\Scal}_D[\bvarphi_{\delta}]} ds(\bx)\right) \\
   = &\Im\left( \sum_{i=1}^{3}\sum_{n=1}^{\infty} \sum_{m=-n}^{n}  |\varphi_{\delta,i}^{n,m}|^2 \overline{\widetilde{e_i^n}} \left(-1/2+\widetilde{\xi_i^n}\right)  \right),
 \end{split}
\end{equation}
where $\widetilde{\xi_i^n}, \widetilde{e_i^n}, n\geq 1, i=1,2,3$ are the eigenvalues the of the operators $\widetilde{\Kcal}^*_D$ and $\widetilde{\Scal}_D$ corresponding to the parameters $(\widetilde{\lambda},\widetilde{\mu})$, respectively. Next we give the estimate of $\Im\left(\overline{\widetilde{e_i^n}} \left(-1/2+\widetilde{\xi_i^n}\right)\right)$ for $i=1,2,3$.
When $i=1$ and $n\geq 1$, one has that
\begin{equation}\label{eq:estimate_eff_1}
  \Im\left(\overline{\widetilde{e_1^n}} \left(-1/2+\widetilde{\xi_1^n}\right)\right) = \Im\left( \frac{n-1}{(2n+1)^2\overline{\widetilde{\mu}}} \right)= \frac{(n-1)}{(2n+1)^2(\epsilon_1^2+\delta^2)\mu_0} \delta.
\end{equation}
When $i=2$ and $n\geq 1$, one can derive that
\begin{equation}\label{eq:estimate_eff_2}
  \Im\left(\overline{\widetilde{e_2^n}} \left(-1/2+\widetilde{\xi_2^n}\right)\right) = \Im \left( 2(n-1)|\widetilde{e_2^n} \widetilde{\mu}|^2\frac{1}{\overline{\widetilde{\mu}}} \right) = \frac{2(n-1)|\widetilde{e_2^n} \widetilde{\mu}|^2 }{(\epsilon_1^2 + \delta^2)\mu_0} \delta.
\end{equation}
When $i=3$ and $n\geq 1$, there holds
\begin{equation}\label{eq:estimate_eff_3}
    \Im\left(\overline{\widetilde{e_2^n}} \left(-1/2+\widetilde{\xi_2^n}\right)\right)  = (\delta \mu_0 d_1 + \delta^3\mu_0 d_2)/d_3,
\end{equation}
where
\begin{equation}
\begin{split}
  d_1=& 4\epsilon_1\epsilon_2\lambda_0\mu_0(n^3-2n^2+2n-1)+\epsilon_1^2\lambda_0^2(2n^3-2n^2+n-1) \\
      & + \epsilon_2^2\mu_0(\lambda_0(4n^2-1)n + 2\mu_0(3n^3-5n^2+5n-2)),
\end{split}
\end{equation}
\begin{equation}
  d_2=(\lambda_0(n-1)+\mu_0(3n-2))(\lambda_0(2n^2+1) + 2\mu_0(n^2-n+1)),
\end{equation}
and
\begin{equation}
  d_3=|(\widetilde{\lambda}+2\widetilde{\mu})(4n^2-1)|^2.
\end{equation}
Since $\delta$ is sufficient small, only the contribution of $\delta \mu_0 d_1/d_3$ should be taken into consideration. Using the strong convexity condition given in \eqref{eq:convex}, one can conclude that, when $n\geq 1$,
\begin{equation}
  d_1>0,
\end{equation}
which is independent of the values of $\epsilon_1$ and $\epsilon_2$. From \eqref{eq:estimate_eff_1}, \eqref{eq:estimate_eff_2}, one can obtain that when $n=1$,
\begin{equation}
  \Im\left(\overline{\widetilde{e_1^1}} \left(-1/2+\widetilde{\xi_1^1}\right)\right) = \Im\left(\overline{\widetilde{e_2^1}} \left(-1/2+\widetilde{\xi_2^1}\right)\right)=0.
\end{equation}
Thus we decompose $\bvarphi_{\delta}$ into two parts, namely
\begin{equation}\label{eq:solution_decom}
  \bvarphi_{\delta}=\bvarphi_{\delta}^{\prime}+\bvarphi_{\delta}^{\prime\prime},
\end{equation}
 where
 \[
 \bvarphi_{\delta}^{\prime\prime} =\sum_{i=1}^2 \sum_{m=-1}^{1} \varphi_{\delta,i}^{1,m} \bkappa_i^{1,m},
 \]
 and
 \[
 \bvarphi_{\delta}^{\prime}= \bvarphi_{\delta}-\bvarphi_{\delta}^{\prime\prime}.
 \]
 From \eqref{eq:estimate_eff_1}, \eqref{eq:estimate_eff_2} and \eqref{eq:estimate_eff_3}, together with \eqref{eq:solution_decom}, there holds that
\begin{equation}\label{eq:energy_main}
  \Im\left( \int_{\partial D} \left(- \frac{1}{2}I + \widetilde{\Kcal}_D^*\right)[\bvarphi_{\delta}] \cdot \overline{\widetilde{\Scal}_D[\bvarphi_{\delta}]} ds(\bx)\right) \approx \delta \norm{\bvarphi_{\delta}^{\prime}}_{\mch^*}^2.
\end{equation}
Combining \eqref{eq:energy_decom}, \eqref{eq:energy_control_A} and \eqref{eq:energy_main}, one has that
\begin{equation}\label{eq:estimate_energy}
  (\delta-\omega)\norm{\bvarphi_{\delta}^{\prime}}_{\mch^*}^2 -\omega \norm{\bvarphi_{\delta}^{\prime\prime}}_{\mch^*}^2 \leq E(\widetilde{\Scal}[\bvarphi_{\delta}]) \leq (\delta+\omega)\norm{\bvarphi_{\delta}^{\prime}}_{\mch^*}^2 + \omega \norm{\bvarphi_{\delta}^{\prime\prime}}_{\mch^*}^2.
\end{equation}

We are ready to present one of the main results of the current article on the novel plasmonic structures and the corresponding resonances.

\begin{thm}\label{thm_resonance}
  Assume $\omega\delta^{-1}\leq h_0$ for a sufficiently small $h_0\in\mathbb{R}_+$. Let $\bu_{\delta}$ be the solution to \eqref{eq:problem_system} corresponding to $\widetilde{\mathbf{C}}$ and $\mathbf{f}$ with $(\widetilde\lambda,\widetilde\mu)$ defined in \eqref{eq:parame_inside}.

 \begin{enumerate}
   \item[(i)] If $\epsilon_2=c_1^n$ or $c_{2,1}^n$ or $c_{2,2}^n$ with $n\geq 2$ or $\epsilon_1=c_3^n$ with $n\geq 1$, where $c_1^n, c_{2,1}^n, c_{2,2}^n$ and $c_3^n$ are given in statement $(2)$ of Theorem \ref{cor_property_A}, then
\begin{equation}
  E(\bu_{\delta})\approx \delta^{-1}
\end{equation}
 as $\delta\rightarrow +0$. That means, resonance occurs with the plasmonic structures described above.

  \item[(ii)]  If $\epsilon_2\neq c_1^n$, $\epsilon_1\neq c_3^n$ for all $n$, and $\epsilon_2\neq c_{2,j}^n$, for all $n$ and $j=1,2$,, then there exists a positive constant $C$ independent of $\delta$ such that
 \begin{equation}
    E(\bu_{\delta})\leq C.
\end{equation}
That means, resonance does not occur with the plasmonic structures described above.
 \end{enumerate}
\end{thm}
\begin{proof}
  Since $\Acal_{\delta}^{\omega}= \Acal_{\delta} + \Tcal_{\delta}^{\omega} = \Acal_{\delta}(I + (\Acal_{\delta})^{-1}\Tcal_{\delta}^{\omega} )$, it follows from \eqref{eq:equation_system} that
\begin{equation}
  \Phi_{\delta}=(I + (\Acal_{\delta})^{-1}\Tcal_{\delta}^{\omega} )^{-1}\Acal_{\delta}^{-1}[\mathcal{F}],
\end{equation}
where
\begin{equation}
  \Phi_{\delta}=
\left[
   \begin{array}{c}
     \bvarphi_{\delta} \\
     \bpsi_{\delta} \\
   \end{array}
 \right]
\quad \mbox{and} \quad
\mathcal{F}=
\left[
  \begin{array}{c}
    \mathbf{F} \\
    \partial_{\bnu}\mathbf{F} \\
  \end{array}
\right].
\end{equation}
We see from \eqref{eq:control_T_s} and the statement (1) in Theorem \ref{cor_property_A} that
\begin{equation}
  \norm{(\Acal_{\delta})^{-1}\Tcal_{\delta}^{\omega}}_{\mathcal{L}(X)}\leq C\delta^{-1}\omega.
\end{equation}
Hence, if $\delta^{-1}\omega$ is sufficiently small, then there holds
\begin{equation}\label{eq:estimate_im}
  \norm{\Phi_{\delta}-\Acal_{\delta}^{-1}[\mathcal{F}]}_{X} \leq C \delta^{-1}\omega \norm{\Acal_{\delta}^{-1}[\mathcal{F}]}_{X}.
\end{equation}

Suppose that $\epsilon_1$ and $\epsilon_2$ are given in $(i)$ of the theorem. Let $(\Acal_{\delta})^{-1}[\mathcal{F}]=(\bvarphi_1,\bpsi_1)^{T}$. Then (2) in Theorem~\ref{cor_property_A} shows that
\begin{equation}
  \norm{\bvarphi_1^{\prime}}_{\mch^*}\geq C |\varphi_i^{n,m}|\delta^{-1},
\end{equation}
and
\begin{equation}
  \norm{\bvarphi_1^{\prime\prime}}_{\mch^*}\leq C,
\end{equation}
where
\begin{equation}
  \varphi_i^{n,m}= (\partial_{\bnu}\mathbf{F})_i^{n,m}- (1/2 + \xi_i^{n})(\Scal^{-1}[\mathbf{F}])_i^{n,m},
\end{equation}
with
\begin{equation}
  \partial_{\bnu}\mathbf{F} =  \sum_{i=1}^{3} \sum_{n=1}^{\infty} \sum_{m=-n}^{n} (\partial_{\bnu}\mathbf{F})_i^{n,m} \bkappa_i^{n,m},
\end{equation}
and
\begin{equation}
  \Scal^{-1}[\mathbf{F}] =  \sum_{i=1}^{3} \sum_{n=1}^{\infty} \sum_{m=-n}^{n}  (\Scal^{-1}[\mathbf{F}])_i^{n,m} \bkappa_i^{n,m}.
\end{equation}
It then follows from \eqref{eq:estimate_im} that
\begin{equation}
  \norm{\bvarphi_{\delta}}_{\mch^*}\geq C( \norm{\bvarphi_1}_{\mch^*} -\delta^{-1}\omega\norm{\Acal_{\delta}^{-1}[\mathcal{F}]}_X) \geq C |\varphi_i^{n,m}|\delta^{-1},
\end{equation}
if $\varphi_i^{n,m}\neq 0$. Thus we obtain from \eqref{eq:estimate_energy} that
\begin{equation}
   E(\bu_{\delta})=E(\widetilde{\Scal}[\bvarphi_{\delta}])\geq C |\varphi_i^{n,m}|^2\delta^{-1}.
\end{equation}

If $\epsilon_1$ and $\epsilon_2$ satisfy the conditions given in $(ii)$ of the theorem, then we can infer from Theorem \ref{cor_property_A}, $(3)$, that
\begin{equation}
  \norm{\Phi_{\delta}}_{X}\leq C \norm{\mathcal{F}}_{Y}.
\end{equation}
Therefore, we have from {\eqref{eq:solution} and \eqref{eq:estimate_energy}}
\begin{equation}
  E(\bu_{\delta})= E(\widetilde{\Scal}[\bvarphi_{\delta}])\leq C \norm{\bvarphi_{\delta}}_{\mch^*}\leq C.
\end{equation}

The proof is complete.

\end{proof}

\begin{rem}

As discussed earlier in the introduction, in the existing studies \cite{AKKY2,DLL,LiLiu2d,LiLiu3d}, the plasmon parameters are always taken to be of the following form,
\[
\widehat\lambda=\epsilon\lambda_0\quad\mbox{and}\quad \widehat\mu=\epsilon\mu_0,
\]
with $\epsilon$ being an appropriate negative number. Clearly, such plasmon parameters break both convexity conditions in \eqref{eq:convex}. For the plasmonic resonances proved in Theorem~\ref{thm_resonance}, it is straightforward to verify that for the plasmon parameters
\[
\widehat\lambda=\epsilon_1\lambda_0\quad\mbox{and}\quad \widehat\mu=\epsilon_2\mu_0,
\]
considered therein,
  \begin{enumerate}
    \item if $\epsilon_2=c_1^n$, they break the first convexity condition in \eqref{eq:convex};
    \item if $\epsilon_2=c_{2,1}^n$, they break the second convexity condition in \eqref{eq:convex};
    \item if $\epsilon_2=c_{2,2}^n$, they break the first convexity condition in \eqref{eq:convex};
    \item if $\epsilon_1=c_3^n$ with $\epsilon_2<0$, they break the first convexity condition in \eqref{eq:convex} and if $\epsilon_1=c_3^n$ with $\epsilon_2\geq0$, they break the second convexity in \eqref{eq:convex}.
  \end{enumerate}
Particularly, it can be verified that the plasmonic structures considered in \cite{AKKY2,DLL,LiLiu2d,LiLiu3d} are included as some special cases in those general ones constructed in Theorem~\ref{thm_resonance}.
  \end{rem}
\begin{rem}
From the proof in theorem \ref{thm_resonance}, it can be directly seen that if we formally take $\omega=0$ in \eqref{eq:problem_system}, then the resonance results in Theorem~\ref{thm_resonance} still hold. That is, we have rigorously verified the quasi-static approximation for the plasmonic resonances. Moreover, according to Theorem~\ref{thm_resonance}, we have provided the condition $\omega/\delta\ll 1$ that ensure the validity of the quasi-static approximation.
\end{rem}

\section{Cloaking due to anomalous localized resonance in elasto-statics}

In this section, we consider the cloaking due to anomalous localized resonance associated with the newly found plasmonic structures in the previous section.  According to the study in \cite{AKKY2,DLL,LiLiu2d,LiLiu3d}, in order to have the localizing and cloaking effects, one should incorporate a core inside the plasmonic inclusion. That is, the concerned plasmonic structure takes a core-shell-matrix form, which we shall describe in the sequel.

First of all, we assume that the region $D$ is a central ball of radius $r_e$, namely $B_{r_e}$, and inside the the ball $B_{r_e}$ there is a concentric ball $B_{r_i}$ with $r_i<r_e$. Inside $B_{r_i}$ the Lam\'e parameters are $(\breve{\lambda},\breve{\mu})$, where
\begin{equation}\label{eq:mcore}
\breve{\lambda}=\epsilon_3\lambda_0\quad\mbox{and}\quad \breve\mu=\epsilon_4\mu_0,
\end{equation}
with $\epsilon_3\in \mathbb{R}$ and  $\epsilon_4\in \mathbb{R}$. The Lam\'e parameters in the shell $B_{r_e}\backslash \overline{B_{r_i}}$ and in the matrix $\mathbb{R}^3\backslash\overline{B_{r_e}}$ are, respectively, $(\widetilde\lambda,\widetilde\mu)$ and $(\lambda_0,\mu_0)$, as defined in \eqref{eq:homo1} and \eqref{eq:parame_inside}. In what follows, we let $\breve{\mathbf{C}}$ signify the elastic tensor where the Lam\'e parameters are specified by $(\breve{\lambda},\breve{\mu})$ and it is same for $\widetilde{\mathbf{C}}$ and $\mathbf{C}_0$, respectively, corresponding to $(\widetilde\lambda,\widetilde\mu)$ and $(\lambda_0,\mu_0)$. Next we introduce another elastic tensor $\mathbf{C}_B$ as follows,
\begin{equation}\label{eq:total_tensor}
  \mathbf{C}_B := \breve{\mathbf{C}} {\chi_{B_{r_i}}} + \widetilde{\mathbf{C}} {\chi_{B_{r_e}\backslash \overline{B_{r_i}}}} + \mathbf{C}_0 {\chi_{\mathbb{R}^3\backslash \overline{B_{r_e}}}}.
\end{equation}
We shall show that by properly choosing $\epsilon_l$, $l=1,2,3,4$, and a suitable forcing source from a certain class, then cloaking can occur due to anomalous localized resonance. As can be seen that the core-shell-matrix structure involves four Lam\'e parameters to be determined. The corresponding analysis becomes much more complicated and delicate than that in the previous section. In what follows, we shall only consider our cloaking study in the quasi-static regime; that is, we simply take the frequency to be zero. Nevertheless, similar to our study in Section 3 on the plasmonic resonances, it is unobjectionable to claim that the quasi-static approximation is justifiable; that is, the result obtained should also hold for diametrically small plasmonic devices with finite frequencies.

Associated with the elastic configuration $\mathbf{C}_B$, the elasto-statics is described by the following Lam\'e system for $\mathbf{u}_\delta\in H_{loc}^1(\mathbb{R}^3)^3$,
\begin{equation}\label{eq:cloaking_prob_origin}
\left\{
  \begin{array}{ll}
    \nabla \cdot \mathbf{C}_B \hat{\nabla}\bu_{\delta}=\mathbf{f}  & \mbox{in} \quad \mathbb{R}^3,\medskip \\
    \bu_{\delta}(\bx)=\mathcal{O}(|\bx|^{-1}) & \mbox{as} \quad |\bx|\rightarrow\infty,
  \end{array}
\right.
\end{equation}
where $\bff\in H^{-1}(\mathbb{R}^2)^3$ is a source function compactly supported in $\mathbb{R}^3\backslash  \overline{B_{r_e}}$ and satisfies
\begin{equation}\label{eq:ds1}
  \int_{\mathbb{R}^3}\mathbf{f}(\bx)\ d\bx=0.
\end{equation}
It is remarked that \eqref{eq:ds1} should be understood in the distributional sense. For technical reasons again, we shall only consider a special, still very general, class of sources as described in the sequel.

First, we introduce the following Newtonian potential $\mathbf{F}$ of the source $\mathbf{f}$ by
\begin{equation}\label{eq:cloaking_poten_F}
  \mathbf{F}(\bx):= \int_{\mathbb{R}^3\backslash \overline{B_{r_e}}} \mathbf{\Gamma}^0(\bx-\by)\bff(\by)d\by, \quad \bx\in \mathbb{R}^3,
\end{equation}
where $\mathbf{\Gamma}^0$ is given in \eqref{eq:fundamentalsolution_0}. Since the source $\mathbf{f}$ is compactly supported in $\mathbb{R}^3\backslash \overline{B_{r_e}}$, there exist $\tau>0$ such that
\begin{equation}\label{eq:cloaking_poten_F_rela}
 \mathcal{L}_{\lambda_0, \mu_0}\mathbf{F} = 0, \quad \bx\in B_{r_e+\tau}\backslash B_{r_e}.
\end{equation}
Using the fact that the eigenfunctions $\bkappa_i^{n,m}$, $i=1,2,3$ given in \eqref{eq:eigenfunctions} form an orthogonal basis on $L^2(\mathbb{S}^2)^3$, the potential $\mathbf{F}$ can be written as
\begin{equation}
  \mathbf{F}(\bx)=\sum_{i=1}^{3} \sum_{n=1}^{\infty}\sum_{m=-n}^{n}\left( f_i^{n,m}  \Scal_{B_{r_e+\tau}}[\bkappa_i^{n,m}](\bx)\right)+c, \quad \bx\in B_{r_e+\tau}\backslash B_{r_e},
\end{equation}
where $c$ is a constant. In what follows, we shall assume that $f_i^{n,m}= 0$ when $i=2, 3$; that is, the Newtonian potential of the source $\mathbf{f}$ only contains the terms involving $\Scal_{B_{r_e+\tau}}[\bkappa_1^{n,m}]$, and doesn't contain any terms involving $\Scal_{B_{r_e+\tau}}[\bkappa_i^{n,m}]$, $i=2,3$. For any fixed $r_1\in\mathbb{R}_+$ and eigenfunction $\bkappa_1^{n,m}$, straightforward calculations yield that
\begin{equation}\label{eq:solution_single_layer}
  \Scal_{B_{r_1}}[\bkappa_1^{n,m}]=
\left\{
  \begin{array}{ll}
    e_1^n r^n/r_1^{n-1}\bkappa_1^{n,m} , & |\bx|\leq r_1,\medskip \\
    e_1^n r_1^{n+2}/r^{n+1}\bkappa_1^{n,m} , & |\bx|> r_1.
  \end{array}
\right.
\end{equation}
 Therefore, by the above assumption on the source term, we actually have that the Newtonian potential of the source $\mathbf{f}$ has the following representation
\begin{equation}\label{eq:source_expansion}
  \mathbf{F}= \sum^{\infty}_{n=1}\sum_{m=-n}^{n} f^{n,m} \frac{r^n}{r_e^{n}} \bkappa_1^{n,m} + c, \quad \bx\in B_{r_e+\tau}\backslash B_{r_e},
\end{equation}
which contains only the first kind eigenfunctions $\bkappa_1^{n,m}$. In \eqref{eq:source_expansion}, $\left( \left( f^{n,m} \sqrt{n} \right)_{m=-n}^{n} \right)_{n=1}^{\infty}\in l^2$.
For the subsequent use, we have by direct calculations that the traction of $\mathbf{F}$ on the boundary $\partial B_{r_e}$ is given by
\begin{equation}\label{eq:source_traction_expansion}
  \partial_{\bnu} \mathbf{F}=\sum^{\infty}_{n=1}\sum_{m=-n}^{n} f^{n,m} \frac{\mu_0(n-1)}{r_e} \bkappa_1^{n,m} \quad \bx\in\partial B_{r_e}.
\end{equation}

Similar to \eqref{eq:def0}, we introduce the following quantity signifying the energy dissipation of the elasto-static system \eqref{eq:cloaking_prob_origin},
\begin{equation}
  E(\bu_{\delta}):=\Im\int_{D} \widehat{\nabla}\bu_{\delta}: \widetilde{\mathbf{C}}\overline{\widehat{\nabla}\bu_{\delta}(\bx)}d\bx.
\end{equation}
We say that cloaking due anomalous localized resonance occurs if the following conditions are satisfied
\begin{equation}\label{eq:cloaking_condition}
  \left\{
    \begin{array}{l}
      \displaystyle{\limsup_{\delta\rightarrow +0} E(\bu_{\delta})\rightarrow +\infty},  \medskip \\
      |\bu_{\delta}(x)|\leq C, \ \ \ |x|>r_0,
    \end{array}
  \right.
\end{equation}
where $\mathbf{u}_\delta\in H^1_{loc}(\mathbb{R}^3)^3$ is the solution to \eqref{eq:cloaking_prob_origin}, and $C, r_0$ are two finite positive constants. By \eqref{eq:cloaking_condition}, one sees that the highly oscillatory behaviour of the resonant field is confined within the region $B_{r_0}$, and hence localizing effect is observed. Moreover, in the physical situation, by scaling the elasto-static system with a factor being $1/\sqrt{E(\mathbf{u}_\delta)}$, one readily verifies that the scaled filed outside $B_{r_0}$ vanishes, and hence the scaled source as well as the plasmonic structure are invisible to the measurement made in $\mathbb{R}^3\backslash\overline{B_{r_0}}$.

Henceforth, we let $\mathcal{L}_{\breve{\lambda}, \breve{\mu}}$, $\partial_{\breve{\bnu}}$, $\breve{\Scal}_{B_{r_e}}$ and $\breve{\Kcal}_{B_{r_e}}^*$, respectively, denote the Lam\'e operator, the associated conormal derivative, the single layer potential operator and the N-P operator corresponding to the Lam\'e parameters $(\breve{\lambda}, \breve{\mu})$. The same notations shall be adopted for the Lam\'e parameters $(\widetilde{\lambda}, \widetilde{\mu})$ . Then it can be readily seen that the elasto-static system \eqref{eq:cloaking_prob_origin} can be equivalently formulated as the following transmission problem,
\begin{equation}\label{eq:cloaking_prob_system}
  \left\{
    \begin{array}{ll}
      \mathcal{L}_{\breve{\lambda},\breve{\mu}}\bu_{\delta}(\bx) =0 , & \mbox{in} \ \ B_{r_i} , \medskip  \\
      \mathcal{L}_{\widetilde\lambda,\widetilde\mu}\bu_{\delta}(\bx) =0 , & \mbox{in} \ \ B_{r_e}\backslash \overline{B_{r_i}} ,\medskip   \\
      \mathcal{L}_{\lambda_0, \mu_0}\bu_{\delta}(\bx) =\bff, &  \mbox{in} \ \ \mathbb{R}^3\backslash \overline{B_{r_e}},\medskip \\
      \bu_{\delta}|_- = \bu_{\delta}|_+, \quad \partial_{\breve{\bnu}}\bu_{\delta}|_- = \partial_{\widetilde{\bnu}}\bu_{\delta}|_+   & \mbox{on} \ \ \partial B_{r_i} ,\medskip\\
      \bu_{\delta}|_- = \bu_{\delta}|_+, \quad  \partial_{\widetilde{\bnu}}\bu_{\delta}|_- = \partial_{\bnu}\bu_{\delta}|_+ & \mbox{on} \; \partial B_{r_e}.
    \end{array}
  \right.
\end{equation}
The solution to the PDE system \eqref{eq:cloaking_prob_system} can be represented by using the following integral ansatz,
\begin{equation}\label{eq:solution_cloaking}
  \bu_{\delta}(\bx)=
  \left\{
    \begin{array}{ll}
      \breve{\Scal}_{B_{r_i}}[\bupsilon](\bx), & \bx\in B_{r_i},\medskip \\
      \widetilde{\Scal}_{B_{r_i}}[\bphi](\bx) + \widetilde{\Scal}_{B_{r_e}}[\bvarphi](\bx), & \bx\in B_{r_e}\backslash \overline{B_{r_i}},\medskip \\
      \Scal_{B_{r_e}}[\bpsi](\bx) + \mathbf{F}(\bx), & \bx\in \mathbb{R}^3\backslash \overline{B_{r_e}},
    \end{array}
  \right.
\end{equation}
where $\mathbf{F}$ is given in \eqref{eq:cloaking_poten_F}. It is obvious that the solution given in \eqref{eq:solution_cloaking} satisfies the first three equalities in \eqref{eq:cloaking_prob_system}. Furthermore, by matching the boundary transmission conditions in \eqref{eq:cloaking_prob_system}, we have
\begin{equation}\label{eq:cloaking_divided}
  \left\{
    \begin{array}{ll}
      \breve{\Scal}_{B_{r_i}}[\bupsilon]=\widetilde{\Scal}_{B_{r_i}}[\bphi] + \widetilde{\Scal}_{B_{r_e}}[\bvarphi], & \mbox{on} \quad \partial B_{r_i},\medskip \\
       \partial_{\breve{\bnu}}\breve{\Scal}_{B_{r_i}}[\bupsilon]|_- = \partial_{\widetilde{\bnu}}(\widetilde{\Scal}_{B_{r_i}}[\bphi] + \widetilde{\Scal}_{B_{r_e}}[\bvarphi])|_+ , & \mbox{on} \quad \partial B_{r_i}, \medskip \\
      \widetilde{\Scal}_{B_{r_i}}[\bphi] + \widetilde{\Scal}_{B_{r_e}}[\bvarphi]= \Scal_{B_{r_e}}[\bpsi] + \mathbf{F}, & \mbox{on} \quad \partial B_{r_e}, \medskip \\
      \partial_{\widetilde{\bnu}}(\widetilde{\Scal}_{B_{r_i}}[\bphi] + \widetilde{\Scal}_{B_{r_e}}[\bvarphi])|_- = \partial_{\bnu}(\Scal_{B_{r_e}}[\bpsi] + \mathbf{F})|_+ , & \mbox{on} \quad \partial B_{r_e}.
    \end{array}
  \right.
\end{equation}
With the help of the jump property in \eqref{eq:jump_single}, \eqref{eq:cloaking_divided} further yields the following integral system,
\begin{equation}\label{eq:cloaking_matrix}
  \left[
    \begin{array}{cccc}
       \breve{\Scal}_{B_{r_i}} & -\widetilde{\Scal}_{B_{r_i}} & -\widetilde{\Scal}_{B_{r_e}} & 0 \\
      -\frac{1}{2} + \breve{\Kcal}_{B_{r_i}}^*  & -\frac{1}{2} - \widetilde{\Kcal}_{B_{r_i}}^* &\partial_{\widetilde{\bnu}_i} \widetilde{\Scal}_{B_{r_e}} & 0 \\
      0 & \widetilde{\Scal}_{B_{r_i}} & \widetilde{\Scal}_{B_{r_e}} & -\Scal_{B_{r_e}} \\
      0 & \partial_{\bnu_e}\Scal_{B_{r_i}} & -\frac{1}{2} + \widetilde{\Kcal}_{B_{r_e}}^* & -\frac{1}{2} - \Kcal_{B_{r_e}}^* \\
    \end{array}
  \right]
\left[
  \begin{array}{c}
    \bupsilon \\
    \bphi \\
    \bvarphi \\
    \bpsi \\
  \end{array}
\right]=
\left[
  \begin{array}{c}
    0 \\
    0 \\
    \mathbf{F} \\
    \partial_{\bnu}\mathbf{F} \\
  \end{array}
\right],
\end{equation}
where $\partial_{\widetilde{\bnu}_i}$ and $ \partial_{\bnu_e}$ signify taking the conormal derivatives on the boundaries of $B_{r_i}$ and $B_{r_e}$, respectively.
Substituting \eqref{eq:solution_single_layer}, \eqref{eq:source_expansion} and \eqref{eq:source_traction_expansion} into \eqref{eq:cloaking_matrix}, along with the use of Theorem~\ref{thm:eigenfunction}, one can conclude that the solutions to \eqref{eq:cloaking_matrix} have the following Fourier series expansions,
\begin{equation}
  \bupsilon=\sum^{\infty}_{n=2}\sum_{m=-n}^{n} \frac{\upsilon^{n,m}}{d^{n}} \bkappa_1^{n,m} + \bupsilon_1,\qquad
  \bphi=\sum^{\infty}_{n=2}\sum_{m=-n}^{n} \frac{\phi^{n,m}}{d^{n}} \bkappa_1^{n,m},
\end{equation}
and 
\begin{equation}
  \bvarphi=\sum^{\infty}_{n=2}\sum_{m=-n}^{n} \frac{\varphi^{n,m}}{d^{n}} \bkappa_1^{n,m}+\bvarphi_1,\qquad
  \bpsi=\sum^{\infty}_{n=2}\sum_{m=-n}^{n} \frac{\psi^{n,m}}{d^{n}} \bkappa_1^{n,m},
\end{equation}
where
\begin{equation}
  \bupsilon_1=\sum_{m=-1}^{1} \frac{-3 f^{1,m} \breve{\mu}}{r_e} \bkappa_1^{1,m},\qquad
  \bvarphi_1=\sum_{m=-1}^{1} \frac{-3 f^{1,m} \widetilde{\mu}}{r_e} \bkappa_1^{1,m},
\end{equation}
\begin{equation}
  \upsilon^{n,m}=-f^{n,m} \mu_0 \breve{\mu} \widetilde{\mu}(2n+1)^3\rho^{n-1},
\end{equation}
\begin{equation}
\phi^{n,m}= f^{n,m} \mu_0 (\breve{\mu} - \widetilde{\mu})\widetilde{\mu}(n-1)(2n+1)^2 \rho^{n-1},
\end{equation}
\begin{equation}
  \varphi^{n,m}= -f^{n,m} \mu_0 \widetilde{\mu}(2n+1)^2((n-1)\breve{\mu} +(n+2)\widetilde{\mu}),
\end{equation}
and
\begin{equation}\label{eq:dn1}
\begin{split}
  \psi^{n,m}=&f^{n,m}(n-1)(2n+1)( -\mu_0(\mu_0-\widetilde{\mu})( (n-1)\breve{\mu} +(n+2)\widetilde{\mu} )\\
  & + \mu_0(\breve{\mu}-\widetilde{\mu})( (n-1)\mu_0 +(n+2)\widetilde{\mu}\rho^2 ) ),\\
  d^{n}=&((n-1)\breve{\mu} +(n+2)\mu_0)( (n-1)\breve{\mu} +(n+2)\widetilde{\mu} ) \\
  &+ (n+2)(n-1)\rho^{2n+1}(\breve{\mu}-\widetilde{\mu})(-\mu_0 + \widetilde{\mu}). 
\end{split}
\end{equation}
Here and also throughout the rest of the paper, $\rho:=r_i/r_e$. Then it follows from \eqref{eq:solution_single_layer} that the solution in the shell $B_{r_e}\backslash \overline{B_{r_i}}$ can be written as
\begin{equation}
   \widetilde{\Scal}_{B_{r_i}}[\bphi] + \widetilde{\Scal}_{B_{r_e}}[\bvarphi]=\sum^{\infty}_{n=2}\sum_{m=-n}^{n} e_1^n \left( \frac{\phi^{n,m}}{d^{n}} \frac{r_i^{n+2}}{r^{n+1}} + \frac{\varphi^{n,m}}{d^{n}}\frac{r^n}{r_e^{n-1}}\right)\bkappa_1^{n,m}+r\bvarphi_1, \quad \bx\in B_{r_e}\backslash \overline{B_{r_i}},
\end{equation}
whereas the solution in $\mathbb{R}^3\backslash \overline{B_{r_e}}$ is given by
\begin{equation}\label{eq:solution_outside}
 \mathbf{F}+ \widetilde{\Scal}_{B_{r_e}}[\bpsi]=\mathbf{F} + \sum^{\infty}_{n=2}\sum_{m=-n}^{n} e_1^n \left( \frac{\psi^{n,m}}{d^{n}}\frac{r_e^{n+2}}{r^{n+1}} \right)\bkappa_1^{n,m}, \quad \bx\in B_{r_e}\backslash \overline{B_{r_i}}.
\end{equation}
We set
\begin{equation}\label{eq:cloaking_parameter}
  \epsilon_2=-\frac{n_0+2}{n_0-1} \quad \mbox{and} \quad \epsilon_4= \left(\frac{n_0+2}{n_0-1}\right)^2,
\end{equation}
with $n_0>1$ properly chosen in what follows (cf. \eqref{eq:ss1}). Here, we emphasize that the parameters $\epsilon_l$, $l=1,2,3,4$, completely determine the plasmonic device according to \eqref{eq:total_tensor}. Indeed, in our subsequent study, the parameters $\epsilon_1$ and $\epsilon_3$ can be any fixed finite real numbers and hence the choice of the parameters in \eqref{eq:cloaking_parameter} fixes the cloaking device for our study. We note that choice of $\epsilon_2$ in \eqref{eq:cloaking_parameter} actually comes from \eqref{eq:choice1} in Theorem~\ref{cor_property_A}. However, we would like to remark that the parameter $n_0$ is chosen to be dependent on $\delta$ (see \eqref{eq:ss1} again), and this is critical in our study. Particularly, it guarantees that the following estimates hold,
\begin{equation}\label{eq:estimate_deno_n_0}
  |d^{n_0}|\approx n_0^2\left(\delta^2+\rho^{2n_0}\right),
\end{equation}
and
\begin{equation}\label{eq:estimate_deno_no_n_0}
  |d^n|\approx \left(\frac{n-n_0}{n_0}\right)^2 \quad \mbox{when} \quad n\neq n_0,
\end{equation}
where $d^n$ is defined in \eqref{eq:dn1}.

Next we can have some preliminary estimate of the energy $E(\bu_{\delta})$. First of all, we divide the solution in the shell  $B_{r_e}\backslash \overline{B_{r_i}}$ into two parts, namely
\begin{equation}
  \widetilde{\Scal}_{B_{r_i}}[\bphi] + \widetilde{\Scal}_{B_{r_e}}[\bvarphi] = H_{n_0} + H_{\widetilde{n_0}},
\end{equation}
where
\[
H_{n_0} =\sum_{m=-n_0}^{n_0} e_1^{n_0} \left( \frac{\phi^{n_0,m}}{d^{n_0}} \frac{r_i^{n_0+2}}{r^{n_0+1}} + \frac{\varphi^{n_0,m}}{d^{n_0}}\frac{r^{n_0}}{r_e^{n_0-1}}\right)\bkappa_1^{n_0,m},
\]
and
\[
 H_{\widetilde{n_0}} = \sum^{\infty}_{n=2,n\neq n_0}\sum_{m=-n}^{n} e_1^n \left( \frac{\phi^{n,m}}{d^{n}} \frac{r_i^{n+2}}{r^{n+1}} + \frac{\varphi^{n,m}}{d^{n}}\frac{r^n}{r_e^{n-1}}\right)\bkappa_1^{n,m}+ r\bvarphi_1.
\]
Then the energy $E(\bu_{\delta})$ can be represented as
\begin{equation}
  \begin{split}
    E(\bu_{\delta}) & =\Im\int_{D} \widehat{\nabla}\bu_{\delta}: \widetilde{\mathbf{C}}\overline{\widehat{\nabla}\bu_{\delta}}d\bx \\
      & =\Im\int_{D} \widehat{\nabla}H_{n_0}: \widetilde{\mathbf{C}}\overline{\widehat{\nabla}H_{n_0}}d\bx + \Im\int_{D} \widehat{\nabla}H_{\widetilde{n_0}}: \widetilde{\mathbf{C}}\overline{\widehat{\nabla}H_{\widetilde{n_0}}}d\bx.
  \end{split}
\end{equation}
Straightforward calculations give
\begin{equation}
  \Im\int_{D} \widehat{\nabla}H_{n_0}: \widetilde{\mathbf{C}}\overline{\widehat{\nabla}H_{n_0}}d\bx \approx \sum_{m=-n_0}^{n_0} \frac{|f^{n_0,m}|^2n_0\delta}{\delta^2 +\rho^{2n_0}},
\end{equation}
and
\begin{equation}
  \Im\int_{D} \widehat{\nabla}H_{\widetilde{n_0}}: \widetilde{\mathbf{C}}\overline{\widehat{\nabla}H_{\widetilde{n_0}}}d\bx \approx \sum^{\infty}_{n=2,n\neq n_0}\sum_{m=-n}^{n} \frac{\delta |f^{n,m}|^2 n^4 n_0 }{(n-n_0)^3}\leq C\delta,
\end{equation}
since $\left( f^{n,m} \sqrt{n} \right)_{n=1, m=-n,\ldots,n}^{\infty}\in l^2$. Then one can conclude that
\begin{equation}\label{eq:cloaking_energy_1}
  E(\bu_{\delta})\approx \sum_{m=-n_0}^{n_0} \frac{|f^{n_0,m}|^2n_0\delta}{\delta^2 +\rho^{2n_0}}.
\end{equation}

We are in a position to present the main theorem of the current section.

\begin{thm}
  Let the Lam\'e tensor $\mathbf{C}_B$ be described in \eqref{eq:total_tensor} with $\epsilon_2$ and $\epsilon_4$ given in \eqref{eq:cloaking_parameter}. Assume the source $\mathbf{f}$ is supported in $B_{r^*}\backslash \overline{B_{r_i}}$ with $r^*=\sqrt{r_e^3/r_i}$ and its Newtonian potential $\mathbf{F}$ is given in \eqref{eq:source_expansion} with $\left( f^{n,m}\sqrt{n} \right)_{n=1, m=-n,\ldots,n}^{\infty}\in l^2$. Then the cloaking due to anomalous localized resonance occurs, namely, the two conditions in \eqref{eq:cloaking_condition} are fulfilled. If the source $\mathbf{f}$ is supported outside $B_{r^*}$, then the anomalous localized resonance does not occur.
\end{thm}
\begin{proof}
First, we prove the non-resonance result. If the source $\mathbf{f}$ is supported outside $B_{r^*}$, then the potential $\mathbf{F}$ given in \eqref{eq:source_expansion} converges on $|\bx|=r^*+ \alpha$ with $r^*=\sqrt{r_e^3/r_i}$ and $\alpha>0$ small enough, namely
\begin{equation}
  \sum^{\infty}_{n=2}\sum_{m=-n}^{n}  \frac{r_e^n}{r_i^n}n|f^{n,m}|^2<C,
\end{equation}
with $C$ a certain positive constant.
From the expression in \eqref{eq:cloaking_energy_1}, one then has
\begin{equation}
  E(\bu_{\delta})\approx \sum_{m=-n_0}^{n_0} \frac{|f^{n_0,m}|^2n_0}{\delta +\frac{\rho^{2n_0}}{\delta}}\leq Cn_0 \sum_{m=-n_0}^{n_0} \frac{r_e^{n_0}}{r_i^{n_0}} |f^{n_0,m}|^2\leq C,
\end{equation}
which means that the anomalous localized resonance does not occur.

Next we prove the result of cloaking due to anomalous localized resonance. Suppose that the source $\mathbf{f}$ is supported inside $B_{r^*}$, then there holds
\[
\limsup_{n\rightarrow \infty}\left(\sum_{m=-n}^{m=n} \frac{|f^{n,m}| }{ r_e^{n}} \right)^{1/n} > 1/\sqrt{\frac{r_e^3}{r_i}},
\]
namely
\begin{equation}\label{eq:source_infity}
 \limsup_{n\rightarrow \infty}\left(\sum_{m=-n}^{m=n} |f^{n,m}|  \right)^2 >  \frac{r_i^n}{r_e^n}.
\end{equation}
Let $n_0\in\mathbb{N}$ be chosen such that
\begin{equation}\label{eq:ss1}
  (r_i/r_e)^{n_0}< \delta \leq (r_i/r_e)^{n_0-1}.
\end{equation}
It is remarked that for sufficiently small $\delta$, one has that $n_0>1$. Then by using \eqref{eq:cloaking_energy_1}, the following estimate of energy $E(\bu_{\delta})$ holds
\begin{equation}\label{eq:cloaking_energy_2}
  E(\bu_{\delta})\approx \frac{r_e^{n_0}}{r_i^{n_0}} \sum_{m=-n_0}^{n_0} |f^{n_0,m}|^2n_0 \geq C\frac{r_e^{n_0}}{r_i^{n_0}} \frac{n_0}{(2n_0+1)} \left( \sum_{m=-n_0}^{n_0} |f^{n_0,m}| \right)^2.
\end{equation}
Finally, from \eqref{eq:source_infity} and \eqref{eq:cloaking_energy_2}, one can obtain that
\begin{equation}
 \sup E(\bu_{\delta})\rightarrow +\infty \quad \mbox{as} \quad \delta\rightarrow +0,
\end{equation}
which means that the first condition of \eqref{eq:cloaking_condition} is satisfied. As for the second one, direct computation yields that
\begin{equation}
  |\psi^{n_0,m}|\approx \delta |f^{n_0,m}| n_0^3.
\end{equation}
Thus, with the help of \eqref{eq:estimate_deno_n_0}, one can conclude that
\begin{equation}
  |\frac{\psi^{n_0,m}}{d^{n_0}}| \approx \frac{|f^{n_0,m}| n_0}{\delta +\frac{\rho^{2n_0}}{\delta}}\leq C \frac{r_e^{n_0}}{r_i^{n_0}}|f^{n_0,m}| n_0.
\end{equation}
For $n\neq n_0$, one has that
\begin{equation}
   |\psi^{n,m}|\approx \frac{n^2|n-n_0| |f^{n,m}|}{n_0},
\end{equation}
which together with \eqref{eq:estimate_deno_no_n_0} implies that for $n\neq n_0$
\begin{equation}
  |\frac{\psi^{n,m}}{d^{n}}| \approx  \frac{n^2 n_0 |f^{n,m}|}{|n-n_0|}.
\end{equation}
Finally, from the expression of the solution outside $B_{r_e}$ in \eqref{eq:solution_outside}, one can obtain that if $|\bx|>r_e^2/r_i$
\begin{equation}
  |\bu_{\delta}|\leq |\mathbf{F}| + C\sum^{\infty}_{n=2}\sum_{m=-n}^{n} n \frac{r_e^{2n}}{r_i^n} \frac{|f^{n,m}|}{r^n}\leq C,
\end{equation}
which is the second condition of \eqref{eq:cloaking_condition}. That is, cloaking due to anomalous localized resonance occurs.

The proof is complete.
\end{proof}

\begin{rem}
  In order to make the phenomenon of cloaking due to anomalous localized resonance occur, we only need to take $\epsilon_2$ and $\epsilon_4$ to be the ones given in \eqref{eq:cloaking_parameter}; that is, only the Lame parameters $\breve{\mu}$ and $\widetilde{\mu}$ have a certain specific form, whereas the Lam\'e parameters $\breve{\lambda}$ and $\widetilde{\lambda}$ are free to choose.
\end{rem}

\section{Concluding remarks}

In this paper, we consider the plasmonic resonance and cloaking due to anomalous localized resonance in linear elasticity. The crucial ingredient is allowing the presence of negative elastic materials, which breaks the strong convexity conditions satisfied by the Lam\'e parameters. The resonant field reveals highly oscillatory behaviour that is manifested by the blowup of the energy dissipation in the underlying elastic system as the loss parameter goes to zero. The choice of the negative Lam\'e parameters for the plasmonic structures is critical for the occurrence of the resonances. Within the spherical geometric setup, we derive all the possible plasmonic structures that can induce resonances. These include the existing ones studied in the literature as a special case in the current article. Furthermore, we prove the plasmonic resonances for those newly found structures within the finite-frequency regime. Prior to our result, the plasmonic resonances have been investigated only for elasto-statics. Hence, our results validate the quasi-static approximation of the plasmonic resonance for diametrically small elastic inclusions. Finally, as an application of the newly found structures, we construct a plasmonic device of the core-shell-matrix form that can induce cloaking due to anomalous localized resonance in the quasi-static regime. This also includes the existing study in the literature as a particularly case.

In our study, we have made a combination use of layer potential technique, spectral analysis, asymptotic analysis as well as Fourier techniques. In principle, the framework developed can be used to derive the plasmonic resonances and cloaking for the elliptic geometry using the elliptical coordinate system. However, one would encounter much more technical and delicate analysis. The newly found plasmonic structures may find interesting applications in other areas such as imaging resolution enhancement. We shall investigate these and other interesting issues in our future work.

\section*{Acknowledgement}

The work was supported by the startup fund and the FRG grants from Hong Kong Baptist University, Hong Kong RGC General Research Fund, No. 12302415, and the NSF grant of China, No. 11371115.

\end{document}